\documentclass[11pt, oneside]{article}   	
\usepackage{geometry}                		
\usepackage[parfill]{parskip}    		
\usepackage{graphicx}				
\usepackage[margin=10pt,font=footnotesize,labelfont=bf]{caption}
\usepackage{subcaption}
\usepackage{float}
\usepackage{amssymb}
\usepackage{amsmath}
\usepackage{bm}
\usepackage{bbm}
\usepackage{mleftright}
\usepackage{todonotes}

\graphicspath{ {images/} }

\usepackage[square,numbers]{natbib}
\usepackage{url}
\newcommand{\bigO}{\mathcal{O}}
\newcommand{\half}{\frac{1}{2}}
\newcommand{\R}{\mathbb{R}}

\newcommand{\N}{\mathbb{N}}

\newcommand{\ddx}{\frac{\mathrm{d}}{\mathrm{d}x}}
\newcommand{\ddy}{\frac{\mathrm{d}}{\mathrm{d}y}}
\newcommand{\pddx}{\frac{\partial}{\partial x}}
\newcommand{\pddy}{\frac{\partial}{\partial y}}

\newcommand{\hdop}{H}
\newcommand{\bighdop}{\mathbb{\hdop}}
\newcommand{\hdopnk}{\hdop_{n,k}}
\newcommand{\hdopnkab}{\hdop_{n,k}^{(a,b)}}

\newcommand{\Wab}{{W^{(a,b)}}}
\newcommand{\hdopmj}{\hdop_{m,j}}

\newcommand{\Dnt}{D^\top_n}

\newcommand{\jac}{{P}}
\newcommand{\genjac}{R}
\newcommand{\genjacnmk}{\genjac_{n-k}}
\newcommand{\genjacmmj}{\genjac_{m-j}}
\newcommand{\genjacw}{w_\genjac}
\newcommand{\jacw}{w_P}
\newcommand{\normgenjac}{\omega_\genjac}
\newcommand{\normjac}{\omega_P}

\newcommand{\element}{\tau}

\newcommand{\FEset}{\mathcal{T}}
\newcommand{\bigW}{\mathbb{W}}

\newcommand{\hdopnkabc}{\hdop_{n,k}^{(a,b,c)}}

\newcommand{\Wabc}{{W^{(a,b,c)}}}

\newcommand{\betaabc}{\beta^{(a,b,c)}}
\newcommand{\bighdopabc}{\bighdop^{(a,b,c)}}
\newcommand{\Wiii}{W^{(1,1,1)}}
\newcommand{\hdopiii}{\hdop^{(1,1,1)}}
\newcommand{\bighdopiii}{{\mathbb{\hdop}^{(1,1,1)}}}

\newcommand{\bighdopooo}{{\mathbb{\hdop}^{(0,0,0)}}}

\newcommand{\laplacewiii}{\Delta_W^{(1,1,1)\to(1,1,1)}}
\newcommand{\laplacewttt}{\Delta_W^{(2,2,2)\to(0,0,0)}}
\newcommand{\laplaceooo}{\Delta^{(0,0,0)\to(2,2,2)}}
\newcommand{\biharmonictwo}{_2\Delta_W^{(2,2,2)\to(2,2,2)}}
\newcommand{\bigWiii}{{\mathbb{W}^{(1,1,1)}}}
\newcommand{\bigWabc}{\mathbb{W}^{(a,b,c)}}

\newcommand{\alphaabcd}{\alpha^{(a,b,c,d)}}
\newcommand{\betaabcd}{\beta^{(a,b,c,d)}}
\newcommand{\hdopnkabcd}{\hdop_{n,k}^{(a,b,c,d)}}
\newcommand{\hdopabcd}{\hdop^{(a,b,c,d)}}
\newcommand{\Wabcd}{{W^{(a,b,c,d)}}}
\newcommand{\hdopmjabcd}{\hdop_{m,j}^{(a,b,c,d)}}
\newcommand{\bighdopabcd}{\bighdop^{(a,b,c,d)}}

\newcommand{\genjact}{\tilde{\genjac}}

\newcommand{\genjactw}{w_{\genjact}}
\newcommand{\normgenjact}{\omega_{\genjact}}

\usepackage{amsthm}

\newtheorem{proposition}{Proposition}
\newtheorem{lemma}{Lemma} 
\newtheorem{theorem}{Theorem} 
\newtheorem{definition}{Definition}
\newtheorem{condition}{Condition}

\newtheorem{corollary}{Corollary}

\def\condref#1{Condition~\ref{cond:#1}}

\def\addtab#1={#1\;&=}

  \def\leqaddtab#1\leq{#1\;&\leq}

\def\pr(#1){\left({#1}\right)}
\def\br[#1]{\left[{#1}\right]}
\def\fbr[#1]{\!\left[{#1}\right]}

\def\ip<#1>{\left\langle{#1}\right\rangle}
\def\iip<#1>{\left\langle\!\langle{#1}\right\rangle\!\rangle}

\def\norm#1{\left\| #1 \right\|}

\def\fpr(#1){\!\pr({#1})}

\def\ceil#1{\left\lceil#1\right\rceil}

\def\mapengine#1,#2.{\mapfunction{#1}\ifx\void#2\else\mapengine #2.\fi }

\def\map[#1]{\mapengine #1,\void.}

\def\mapenginesep_#1#2,#3.{\mapfunction{#2}\ifx\void#3\else#1\mapengine #3.\fi }

\def\mapsep_#1[#2]{\mapenginesep_{#1}#2,\void.}

\def\vcbr[#1]{\pr(#1)}

\def\bvect[#1,#2]{
{
\def\dots{\cdots}
\def\mapfunction##1{\ | \  ##1}
	\sopmatrix{
		 \,#1\map[#2]\,
	}
}
}

\def\vect[#1]{
{\def\dots{\ldots}
	\vcbr[{#1}]
}}

\def\vectt[#1]{
{\def\dots{\ldots}
	\vect[{#1}]^{\top}
}}

\def\Vectt[#1]{
{
\def\mapfunction##1{##1 \cr} 
\def\dots{\vdots}
	\begin{pmatrix}
		\map[#1]
	\end{pmatrix}
}}

\def\R{{\mathbb R}}

\def\D{{\rm d}}

\def\tF_#1{{\tt F}_{#1}}

\def\tFC_#1{{\tt T}_{#1}}

\def\secref#1{Section~\ref{Section:#1}}

\def\appref#1{Appendix~\ref{Appendix:#1}}

\def\qqand{\qquad\hbox{and}\qquad}
\def\qfor{\quad\hbox{for}\quad}

\def\elllRpz_#1{\ell_{#1{\rm z}}^{(\lambda,R),p}}

\def\sopmatrix#1{\begin{pmatrix}#1\end{pmatrix}}

\def\Problem#1#2\par{\begin{problem}\label{pb:#1} #2\end{problem}}
\def\Theorem#1#2\par{\begin{theorem}\label{th:#1} #2\end{theorem}}
\def\Conjecture#1#2\par{\begin{conjecture}\label{conj:#1} #2\end{conjecture}}
\def\Proposition#1#2\par{\begin{proposition}\label{prop:#1} #2\end{proposition}}
\def\Definition#1#2\par{\begin{definition}\label{def:#1} #2\end{definition}}
\def\Corollary#1#2\par{\begin{corollary}\label{cr:#1} #2\end{corollary}}
\def\Lemma#1#2\par{\begin{lemma}\label{lm:#1} #2\end{lemma}}
\def\Example#1#2\par{\begin{example}\label{ex:#1} #2\end{example}}
\def\Remark #1\par{\begin{remark*}#1\end{remark*}}

\def\Proof{\begin{proof}}
\def\mqed{\end{proof}}

\def\Figuretwow[#1,#2]#3#4\par{
\begin{figure}[tb]
\begin{center}{
\includegraphics[width=#3]{Figures/#1}\includegraphics[width=#3]{Figures/#2}}
\end{center}
\caption{#4}\label{fig:#1} 
\end{figure}
}

\title{Sparse spectral and $p$-finite element methods for partial differential equations on disk slices and trapeziums}
\author{Ben Snowball, Sheehan Olver}

\begin{document}

\maketitle

\begin{abstract}
Sparse spectral methods for solving partial differential equations have been derived in recent years using hierarchies of classical orthogonal polynomials on intervals, disks, and triangles. In this work we extend this methodology to a hierarchy of non-classical orthogonal polynomials on disk slices and trapeziums. This builds on the observation that sparsity is guaranteed due to the boundary being defined by an algebraic curve, and that the entries of partial differential operators can be determined using formulae in terms of (non-classical) univariate orthogonal polynomials. We apply the framework to solving the Poisson, variable coefficient Helmholtz, and Biharmonic equations. In this paper we focus on constant Dirichlet boundary conditions, as well as zero Dirichlet and Neumann boundary conditions, with other types of boundary conditions requiring future work.
 \end{abstract}

\section{Introduction}

This paper develops sparse spectral methods for solving linear partial differential equations on a special class of geometries that includes disk slices and trapeziums.  
More precisely, we consider the solution of partial differential equations on the domain
\begin{align*}
	\Omega := \{(x,y) \in \R^2 \quad | \quad \alpha < x < \beta, \: \gamma \rho(x) < y < \delta \rho(x)\}
\end{align*}
where  either of the following conditions hold:

\begin{condition}\label{cond:trap}
 \(\rho\) is a degree 1 polynomial.
 \end{condition}
 
 \begin{condition}\label{cond:disk}
\(\rho\) is the square root of a non-negative degree \(\le\) 2 polynomial, \(-\gamma = \delta > 0\).
\end{condition}

For simplicity of presentation we focus on the disk-slice, where $\rho(x) = \sqrt{1-x^2}$, $(\alpha,\beta) \subset (0, 1)$, and $(\gamma, \delta)  = (-1,1)$, and discuss an extension to other geometries in the appendix (including the half-disk and trapeziums). 

We show that partial differential equations become sparse linear systems when viewed as acting on expansions involving a family of orthogonal polynomials (OPs) that  generalise Jacobi polynomials, mirroring the ultraspherical spectral method for ordinary differential equations \cite{olver2013fast} and its analogue on the disk \cite{vasil2016tensor} and triangle \cite{olver2018recurrence,olver2019triangle}.  On the disk-slice the family of weights we consider are of the form
$$
\Wabc(x,y) = (\beta - x)^a \: (x - \alpha)^b \: (1-x^2-y^2)^c, \qfor \alpha \leq x \leq \beta,\quad -\rho(x) \leq y \leq \rho(x).
$$
The corresponding OPs denoted $\hdopnkabc(x,y)$, where $n$ denotes the polynomial degree, and $0 \le k \le n$. We define these to be orthogonalised lexicographically, that is,
$$
\hdopnkabc(x,y) = C_{n,k} x^{n-k} y^k + (\hbox{lower order terms})
$$
where $C_{n,k} \neq 0$ and ``lower order terms''   includes degree $n$ polynomials of the form $x^{n-j} y^j$ where $j < k$. The precise normalization arises from their definition in terms of one-dimensional OPs in Definition~\ref{def:OPconstruction}. 

Sparsity comes from expanding the domain and range of an operator using different choices of the parameters $a$, $b$ and $c$. Whereas the sparsity pattern and entries derived in \cite{olver2018recurrence,olver2019triangle} for equations on the triangle and \cite{vasil2016tensor} for equations on the disk results from manipulations of Jacobi polynomials, in the present work we use a more general integration-by-parts argument to deduce the sparsity structure, alongside careful use of the Christoffel--Darboux formula \cite[18.2.2]{DLMF} and quadrature rules to determine the entries. In particular, by exploiting the connection with one-dimensional orthogonal polynomials we can construct discretizations of general partial differential operators of size $p(p-1)/2 \times p(p-1)/2$ in $O(p^3)$ operations, where $p$ is the total polynomial degree. This compares favourably to $O(p^6)$ operations if one proceeds na\"ively. Furthermore, we use this framework to derive sparse $p$-finite element methods that are analogous to those of Beuchler and Sch\"oberl on tetrahedra \cite{beuchler2006new}, see also work by Li and Shen \cite{li2010optimal}.


Here is an overview of the paper:  

\noindent \secref{OPs}: We present our procedure to gain a (two-parameter) family of 2D orthogonal polynomials (OPs) on the disk-slice domain, by combining 1D OPs on the interval, to form 2D OPs on the disk. 

\noindent\secref{PDOs}: We demonstrate that these families will lead to sparse operators, including Jacobi operators representing multiplication by $x$ and $y$, and partial differential operators. We present a method involving use of the Christoffel--Darboux formula \cite[18.2.2]{DLMF} to obtain the recurrence coefficients for the non-classical 1D OPs, allowing us to exactly use the 1D quadrature rules we present to calculate the non-zero entries to the sparse operators.

\noindent\secref{Computation}: We discuss computational issues, in particular, how to realise the results of the preceding sections in practice. We present a method for explicitly deriving the recurrence coefficients of the non-classical 1D OPs we detail in \secref{OPs}. We derive a quadrature rule on the disk-slice that can be used to expand a function in the OP basis up to a given order. Further, we implement function evaluation using the coefficients of the expansion of a given function using the Clenshaw algorithm.

\noindent\secref{Examples}: We demonstrate the proposed technique for solving Poisson, Helmholtz, and Biharmonic equations on the disk-slice.  

\noindent\appref{PFEM}: We use the procedure to construct sparse $p$-finite element methods. This lays the groundwork for a future $hp$-finite element method in a disk, where the elements capture the circular geometry precisely.

\noindent\appref{HalfDisk}: We discuss extension to the special case of end-disk-slices (e.g., half disks).

\noindent\appref{Trapezium}: We discuss extension to trapezia.

\section{Orthogonal polynomials on the disk-slice and the trapezium}\label{Section:OPs}

In this section we outline the construction and some basic properties of $\hdopnkabcd(x,y)$. The symmetry in the weight allows us to express the polynomials in terms of 1D OPs, and deduce certain properties such as recurrence relationships. 

\subsection{Explicit construction}

We can construct 2D orthogonal polynomials on $\Omega$ from 1D orthogonal polynomials on the intervals \([\alpha,\beta]\) and \([\gamma,\delta]\). 

\begin{proposition}[{\cite[p55--56]{dunkl2014orthogonal}}]\label{prop:construction}
Let \(w_1 : (\alpha,\beta) \: \to \R\), \(w_2 : (\gamma,\delta) \: \to \R\) be weight functions with \(\alpha,\beta,\gamma,\delta \in \R\), and let \(\rho \: : \: (\alpha,\beta) \: \to (0,\infty)\) be such that either \condref{trap} or \condref{disk} with \(w_2\) being an even function hold.
$\forall$, $n = 0,1,2,\dots, $ let $\{p_{n,k}\}$ be polynomials orthogonal with respect to the weight $\rho(x)^{2k+1} w_1(x)$ where $0 \le k \le n$, and $\{q_{n}\}$ be polynomials orthogonal with respect to the weight $w_2(x)$. Then the 2D polynomials defined on $\Omega$
$$
\hdopnk(x,y) := p_{n-k,k}(x) \: \rho(x)^k \: q_k\fpr(\frac{y}{\rho(x)}) \qquad\hbox{for} \qquad 0 \le k \le n, \: n = 0,1,2,\dots
$$
are orthogonal polynomials with respect to the weight \(W(x,y) := w_1(x) \: w_2\fpr(\frac{y}{\rho(x)}) \) on $\Omega$. 
\end{proposition}

For disk slices and trapeziums, we specialise Proposition \ref{prop:construction} in the following definition. First we introduce notation for two families of univariate OPs.
\begin{definition}\label{def:OPconstruction}
Let $\genjacw^{(a,b,c)}(x)$ and $\jacw^{(a,b)}(x)$ be two weight functions on the intervals $(\alpha, \beta)$ and $(\gamma, \delta)$ respectively, given by:
\begin{align*}
\begin{cases}
\genjacw^{(a,b,c)}(x) &:= (\beta - x)^a \: (x - \alpha)^{b} \: \rho(x)^{c} \\
\jacw^{(a,b)}(x) &:= (\delta-x)^a \: (x - \gamma)^b
\end{cases}
\end{align*}
and define the associated inner products by:
\begin{align}
	\ip< p, \: q >_{\genjacw^{(a,b,c)}} &:= \frac{1}{\normgenjac^{(a,b,c)}} \: \int_\alpha^\beta p(x) \: q(x) \: \genjacw^{(a,b,c)}(x) \: \D x \label{eqn:ipgenjac} \\
	\ip<p, q>_{\jacw^{(a,b)}} &:= \frac{1}{\normjac^{(a,b)}} \: \int_\gamma^\delta p(y) \: q(y) \: \jacw^{(a,b)}(y)\: \D y \label{eqn:ipjac}
\end{align}
where
\begin{align}
	\normgenjac^{(a,b,c)} := \int_\alpha^\beta \: \genjacw^{(a,b,c)}(x) \: \D x, \quad \normjac^{(a,b)} := \int_\gamma^\delta \: \jacw^{(a,b)}(y) \: \D y. \label{eqn:ipnormalisations}
\end{align}
Denote the three-parameter family of orthonormal polynomials on $[\alpha,\beta]$ by $\{\genjac_n^{(a,b,c)}\}$, orthonormal with respect to the inner product defined in (\ref{eqn:ipgenjac}), and the two-parameter family of orthonormal polynomials on $[\gamma,\delta]$ by $\{\jac_n^{(a,b)}\}$, orthonormal with respect to the inner product defined in (\ref{eqn:ipjac}).
\end{definition}

\begin{definition}\label{def:constuction}
Define the four-parameter 2D orthogonal polynomials via:
\begin{align*}
	\hdopnk^{(a,b,c,d)}(x,y) := \genjacnmk^{(a, b, c+d+2k+1)}(x) \: \rho(x)^k \: \jac_k^{(d,c)}\fpr(\frac{y}{\rho(x)}), \quad (x,y) \in \Omega, 
\end{align*}
\end{definition}

$\{\hdopnk^{(a,b,c,d)}\}$ are orthogonal with respect to the weight
\begin{align*}
	W^{(a,b,c,d)}(x,y) := \genjacw^{(a,b,c+d)}(x) \: \jacw^{(d,c)}\fpr(\frac{y}{\rho(x)}), \quad (x,y) \in \Omega,
\end{align*}
assuming that either \condref{trap} or \condref{disk} with  \(\jacw^{(a,b)}\) being an even function (i.e. $a = b$, and we can hence denote the weight as $\jacw^{(a)}(x) = w^{(a,a)}_P(x) =  (\delta-x^2)^a$) hold.
That is,
\begin{align*}
	\ip< \hdopnk^{(a,b,c,d)}, \: \hdopmj^{(a,b,c,d)} >_{W^{(a,b,c,d)}} &= \normgenjac^{(a,b,c+d+2k+1)} \: \normjac^{(d,c)} \: \delta_{n,m} \: \delta_{k,j},
\end{align*}
where for $f, g : \Omega \to \R$ the inner product is defined as 
\begin{align*}
	\ip< f, \: g >_{W^{(a,b,c,d)}} := \iint_\Omega \: f(x,y) \: g(x,y) \: W^{(a,b,c,d)}(x,y) \: \D y \: \D x.
\end{align*}

We can see that they are indeed orthogonal using the change of variable $t = \frac{y}{\rho(x)}$, for the following normalisation:
\begin{align}
	&\ip< \hdopnkabcd, \: \hdopmjabcd >_\Wabcd \\
	&\quad \quad \quad =  \iint_\Omega \: \Big[ \genjacnmk^{(a,b,c+d+2k+1)}(x) \: \genjacmmj^{(a,b,c+d+2j+1)}(x) \: \rho(x)^{k+j}  \nonumber \\
		&\quad \quad \quad \quad \quad \quad \cdot \: \jac_k^{(d,c)}\fpr(\frac{y}{\rho(x)}) \: \jac_j^{(d,c)}\fpr(\frac{y}{\rho(x)})  \: \Wabcd(x,y) \Big] \: \D y \: \D x \nonumber \\
	&\quad \quad \quad = \Big( \int_\alpha^\beta \: {\genjacnmk^{(a,b,c+d+2k+1)}}(x) \: {\genjacmmj^{(a,b,c+d+2j+1)}}(x) \: \genjacw^{(a,b,c+d+k+j+1)}(x) \: \D x \Big) \nonumber \\
	&\quad \quad \quad \quad \quad \quad \cdot \: \Big( \int_\gamma^\delta \: {\jac_k^{(d,c)}}(t) \:  {\jac_j^{(d,c)}}(t) \: \jacw^{(d,)}(t) \: \D t \Big) \nonumber \\
	&\quad \quad \quad = \normjac^{(d,c)} \: \delta_{k,j} \: \int_\alpha^\beta \: {\genjacnmk^{(a,b,c+d+2k+1)}}(x) \: {\genjac_{m-k}^{(a,b,c+d+2k+1)}}(x) \: \genjacw^{(a,b,c+d+2k+1)}(x) \: \D x \nonumber \\ 
	&\quad \quad \quad = \normgenjac^{(a,b,c+d+2k+1)} \: \normjac^{(d,c)} \: \delta_{n,m} \: \delta_{k,j}. \label{eqn:normhdop}
\end{align}

For the disk-slice, the weight $\Wabc(x,y) = (\beta - x)^a \: (x - \alpha)^b \: (1-x^2-y^2)^c$ results from setting:
\begin{align*}
\begin{cases}
	(\alpha,\beta) &\subset (0,1) \\
	(\gamma,\delta) &:= (-1,1) \\
	\rho(x) &:= (1-x^2)^{\half}
\end{cases}
\end{align*}
so that
\begin{align*}
\begin{cases}
	\genjacw^{(a,b,c)}(x) :=  (\beta - x)^a \: (x - \alpha)^b \: \rho(x)^{c} \\
	\jacw^{(c)}(x) := (1-x)^c \: (1+x)^c = (1-x^2)^c.
\end{cases}
\end{align*}
Note here we can simply remove the need for including a fourth parameter $d$. The 2D OPs orthogonal with respect to the weight above on the disk-slice $\Omega$ are then given by:
\begin{align}\label{eq:diskpolys}
	\hdopnkabc(x,y) := \genjacnmk^{(a, b, 2c+2k+1)}(x) \: \rho(x)^k \: \jac_k^{(c,c)}\fpr(\frac{y}{\rho(x)}), \quad (x,y) \in \Omega
\end{align}
In this case the weight $\jacw(x)$ is an ultraspherical weight, and the corresponding OPs are the normalized Jacobi polynomials $\{\jac_n^{(b, b)}\}$, while the weight $w_R(x)$ is non-classical (it is in fact semi-classical, and is equivalent to a generalized Jacobi weight \cite[\S5]{magnus1995painleve}).


\subsection{Jacobi matrices}

We can express the three-term recurrences associated with $\genjac_n^{(a,b,c)}$ and $\jac_n^{(d,c)}$ as 
\begin{align}
	x \genjac_n^{(a,b,c)}(x) &= \beta_n^{(a,b,c)} \genjac_{n+1}^{(a,b,c)}(x) + \alpha_n^{(a,b,c)} \genjac_n^{(a,b,c)}(x) + \beta_{n-1}^{(a,b,c)} \genjac_{n-1}^{(a,b,c)}(x) \label{eqn:Hrecurrence} \\
	y \jac_n^{(d,c)}(y) &= \delta_n^{(d,c)} \jac_{n+1}^{(d,c)}(y) + \gamma_{n}^{(d,c)} \jac_{n}^{(d,c)}(y) + \delta_{n-1}^{(d,c)} \jac_{n-1}^{(d,c)}(y). \label{eqn:Precurrence}
\end{align}
Of course, for the disk-slice case, we have that $c = d$ and $\gamma_{n}^{(c,c)} = 0 \quad \forall n=0,1,2,\dots$. We can use (\ref{eqn:Hrecurrence}) and (\ref{eqn:Precurrence}) to determine the 2D recurrences for $\hdopnkabcd(x,y)$. Importantly, we can deduce sparsity in the recurrence relationships:

\begin{lemma}
$\hdopnkabcd(x,y)$ satisfy the following 3-term recurrences:
\begin{align*}
x \hdopnkabcd(x,y) &= \alphaabcd_{n,k,1} \: \hdop_{n-1, k}^{(a,b,c,d)}(x, y) + \alphaabcd_{n,k,2} \: \hdop_{n, k}^{(a,b,c,d)}(x, y) + \alphaabcd_{n+1,k,1} \: \hdop_{n+1, k}^{(a,b,c)}(x, y), \\
y \hdopnkabcd(x,y) &= \betaabcd_{n,k,1} \: \hdop_{n-1, k-1}^{(a,b,c,d)}(x, y) + \betaabcd_{n,k,2} \: \hdop_{n-1, k}^{(a,b,c,d)}(x, y) + \betaabcd_{n,k,3} \: \hdop_{n-1, k+1}^{(a,b,c,d)}(x, y) \nonumber \\
		& \quad \quad + \betaabcd_{n,k,4} \: \hdop_{n, k-1}^{(a,b,c,d)}(x, y) + \betaabcd_{n,k,5} \: \hdop_{n, k}^{(a,b,c,d)}(x, y) + \betaabcd_{n,k,6} \: \hdop_{n, k+1}^{(a,b,c,d)}(x, y) \nonumber \\
		& \quad \quad + \betaabcd_{n,k,7} \: \hdop_{n+1, k-1}^{(a,b,c,d)}(x, y) + \betaabcd_{n,k,8} \: \hdop_{n+1, k}^{(a,b,c,d)}(x, y) + \betaabcd_{n,k,9} \: \hdop_{n+1, k+1}^{(a,b,c,d)}(x, y),
\end{align*}
for \((x,y) \in \Omega\), where
\begin{align*}
	\alphaabcd_{n,k,1} &:= \beta_{n-k-1}^{(a, b+k+\half)}, \qquad \alphaabcd_{n,k,2} := \alpha_{n-k}^{(a, b+k+\half)} \\
	\betaabcd_{n,k,1} &:= \delta_{k-1}^{(d,c)} \: \ip<\genjacnmk^{(a, b, c+d+2k+1)}, \: \genjacnmk^{(a, b, c+d+2k-1)}>_{\genjacw^{(a, b, c+d+2k+1)}} \\
	\betaabcd_{n,k,2} &:= \gamma_k^{(d,c)} \: \ip<\genjacnmk^{(a, b, c+d+2k+1)}, \: \rho(x) \genjac_{n-k-1}^{(a, b, c+d+2k+1)}>_{\genjacw^{(a, b, c+d+2k+1)}} \\
	\betaabcd_{n,k,3} &:= \delta_{k}^{(d,c)} \: \ip<\genjacnmk^{(a, b, c+d+2k+1)}, \genjac_{n-k-2}^{(a, b, c+d+2k+3)}>_{\genjacw^{(a, b, c+d+2k+3)}} \\
	\betaabcd_{n,k,4} &:= \delta_{k-1}^{(d,c)} \: \ip<\genjacnmk^{(a, b, c+d+2k+1)}, \: \genjac_{n-k+1}^{(a, b, c+d+2k-1)}>_{\genjacw^{(a, b, c+d+2k+1)}} \\
	\betaabcd_{n,k,5} &:= \gamma_k^{(d,c)} \: \ip<\genjacnmk^{(a, b, c+d+2k+1)}, \: \rho(x) \genjac_{n-k}^{(a, b, c+d+2k+1)}>_{\genjacw^{(a, b, c+d+2k+1)}} \\
	\betaabcd_{n,k,6} &:= \delta_{k}^{(d,c)} \: \ip<\genjacnmk^{(a, b, c+d+2k+1)}, \genjac_{n-k-1}^{(a, b, c+d+2k+3)}>_{\genjacw^{(a, b, c+d+2k+3)}} \\
	\betaabcd_{n,k,7} &:= \delta_{k-1}^{(d,c)} \: \ip<\genjacnmk^{(a, b, c+d+2k+1)}, \: \genjac_{n-k+2}^{(a, b, c+d+2k-1)}>_{\genjacw^{(a, b, c+d+2k+1)}} \\
	\betaabcd_{n,k,8} &:= \gamma_k^{(d,c)} \: \ip<\genjacnmk^{(a, b, c+d+2k+1)}, \: \rho(x) \genjac_{n-k+1}^{(a, b, c+d+2k+1)}>_{\genjacw^{(a, b, c+d+2k+1)}} \\
	\betaabcd_{n,k,9} &:= \delta_{k}^{(d,c)} \: \ip<\genjacnmk^{(a, b, c+d+2k+1)}, \genjacnmk^{(a, b, c+d+2k+3)}>_{\genjacw^{(a, b, c+d+2k+3)}}. 
\end{align*}

\end{lemma}

\begin{proof}
The 3-term recurrence for multiplication by $x$ follows from equation (\ref{eqn:Hrecurrence}). For the recurrence for multiplication by $y$, since $\{\hdopmjabcd\}$ for $m = 0,\dots,n+1$, $j = 0,\dots,m$ is an orthogonal basis for any degree $n+1$ polynomial, we can expand $y \: \hdopnkabcd(x,y) = \sum_{m=0}^{n+1} \sum_{j=0}^m c_{m,j} \: \hdopmjabcd(x,y)$. These coefficients are given by
\begin{align*}
	c_{m,j} = {\ip< y \: \hdopnkabcd, \hdopmjabcd >_{\Wabcd}}{\norm{\hdopmjabcd}^{-2}_{\Wabcd}}.
\end{align*}
Recall from equation (\ref{eqn:normhdop}) that $\norm{\hdopmjabcd}_{\Wabcd}^2 = \normgenjac^{(a,b,c+d+2j+1)} \: \normjac^{(d,c)}$. Then for $m = 0,\dots,n+1$, $j = 0,\dots,m$, using the change of variable $t = \frac{y}{\rho(x)}$:
\begin{align*}
	&\ip<y \hdopnkabcd, \hdopmjabcd>_\Wabcd \\
	&\quad \quad =  \iint_\Omega \hdopnkabcd(x,y) \: \hdopmjabcd(x,y) \: y \: \Wabcd(x,y) \: dy \: dx \\
	&\quad \quad = \Big( \int^\beta_\alpha \genjacnmk^{(a, b, c+d+2k+1)}(x) \: \genjacmmj^{(a, b, c+d+2j+1)}(x) \: \rho(x)^{k+j+2} \: \genjacw^{(a, b,c+d)}(x) \: \D x \Big) \nonumber \\
	&\quad \quad \quad \quad \quad\cdot \: \Big( \int^\delta_{\gamma} \jac_k^{(d,c)}(t) \: \jac_j^{(d,c)}(t) \: t \: \jacw^{(d,c)}(t) \: \D t \Big) \\
	&\quad \quad = \Big( \int^\beta_\alpha \genjacnmk^{(a, b, c+d+2k+1)}(x) \: \genjacmmj^{(a, b, c+d+2j+1)}(x) \: \genjacw^{(a, b,c+d+k+j+2)}(x) \: \D x \Big) \nonumber \\
	&\quad \quad \quad \quad \quad\cdot \: \Big( \int^\delta_{\gamma} \jac_k^{(d,c)}(t) \: \jac_j^{(d,c)}(t) \: t \: \jacw^{(d,c)}(t) \: \D t \Big)
\end{align*}
\begin{align*}
	=
	\begin{cases}
    		\delta_k^{(d,c)} \: \normjac^{(d,c)} \: \normgenjac^{(a, b, c+d+2k+3)} \: \ip<\genjacnmk^{(a, b, c+d+2k+1)}, \genjac_{m-k-1}^{(a, b, c+d+2k+3)}>_{\genjacw^{(a, b, c+d+2k+3)}} \quad& \text{if } j = k+1 \\
		\gamma_k^{(d,c)} \: \normjac^{(d,c)} \: \normgenjac^{(a, b, c+d+2k+1)} \: \ip<\genjacnmk^{(a, b, c+d+2k+1)}, \: \rho(x) \genjac_{m-k}^{(a, b, c+d+2k+1)}>_{\genjacw^{(a, b, c+d+2k+1)}} \quad& \text{if } j = k \\
		\delta_{k-1}^{(d,c)} \: \normjac^{(d,c)} \: \normgenjac^{(a, b, c+d+2k-1)} \: \ip<\genjacnmk^{(a, b, c+d+2k-1)}, \: \rho(x)^2 \genjac_{m-k+1}^{(a, b, c+d+2k-1)}>_{\genjacw^{(a, b, c+d+2k-1)}} \quad& \text{if } j = k-1 \\
		0 & \text{otherwise}
      	\end{cases}
\end{align*}
where, by orthogonality,
\begin{align*}
	\ip<\genjacnmk^{(a, b, c+d+2k+1)}, \genjac_{m-k-1}^{(a, b, c+d+2k+3)}>_{\genjacw^{(a, b, c+d+2k+3)}} &= 0 \qfor m < n-1, \\
	\ip<\genjacnmk^{(a, b, c+d+2k+1)}, \: \rho(x)^2 \genjac_{m-k+1}^{(a, b, c+d+2k-1)}>_{\genjacw^{(a, b, c+d+2k-1)}} &= 0 \qfor m < n-1.
\end{align*}

Finally, if \condref{trap} holds we have that 
$$
\ip<\genjacnmk^{(a, b, c+d+2k+1)}, \: \rho(x) \genjac_{m-k}^{(a, b, c+d+2k+1)}>_{\genjacw^{(a, b, c+d+2k+1)}} = 0 \qfor m < n-1.
$$
If \condref{disk} holds we have that $\gamma^{(d,c)}_{k} = \gamma^{(c,c)}_{k} \equiv 0$ for any $k$.

\end{proof}

Three-term recurrences lead to Jacobi operators that correspond to multiplication by $x$ and $y$. Define, for $n=0,1,2,\dots$: 
\begin{align*}
\bighdopabcd_n := \begin{pmatrix}
		\hdopabcd_{n,0}(x,y) \\
		\vdots \\
		\hdopabcd_{n,n}(x,y)
	\end{pmatrix} \in \R^{n+1}, 
\quad \quad 
\bighdopabcd := \begin{pmatrix}
		\bighdopabcd_0 \\
		\bighdopabcd_1 \\
		\bighdopabcd_2 \\
		\vdots \\
	\end{pmatrix}
\end{align*}
and set $J_x^{(a,b,c,d)}, J_y^{(a,b,c,d)}$ as the Jacobi matrices corresponding to
\begin{align}
J_x^{(a,b,c,d)} \: \bighdopabcd(x,y) = x \: \bighdopabcd(x,y), \quad J_y^{(a,b,c,d)} \: \bighdopabcd(x,y) = y \: \bighdopabcd(x,y).
\label{eqn:jacobimatricesdefinition}
\end{align}
The matrices $J_x^{(a,b,c,d)}, J_y^{(a,b,c,d)}$ act on the coefficients vector of a function's expansion in the $\{\hdopnkabcd\}$ basis. For example, let $a, b$ be general parameters and a function $f(x,y)$ defined on $\Omega$ be approximated by its expansion $f(x,y) = \bighdopabcd(x,y)^\top \mathbf{f}$. Then $x \: f(x,y)$ is approximated by $\bighdopabcd(x,y)^\top {J_x^{(a,b,c,d)\top}} \mathbf{f}$. In other words, ${J_x^{(a,b,c,d)\top}} \mathbf{f}$ is the coefficients vector for the expansion of the function $(x,y) \mapsto x \: f(x,y)$ in the  $\{\hdopnkabcd\}$ basis. Further, note that $J_x^{(a,b,c,d)}, J_y^{(a,b,c,d)}$ are banded-block-banded matrices:

\begin{definition}
A block matrix $A$ with blocks $A_{i,j}$ has block-bandwidths $(L,U)$ if $A_{i,j} = 0$ for $- L \leq j-i \leq U$, and sub-block-bandwidths $(\lambda, \mu)$ if all blocks $A_{i,j}$ are banded with bandwidths $(\lambda,\mu)$. A matrix where the block-bandwidths and sub-block-bandwidths are small compared to the dimensions is referred to as a banded-block-banded matrix. 
\end{definition}

For example, $J_x^{(a,b,c,d)}, J_y^{(a,b,c,d)}$ are block-tridiagonal (block-bandwidths $(1,1)$):
\begin{align*}
J_{x/y}^{(a,b,c,d)} &= \begin{pmatrix}
		B^{x/y}_0 & A^{x/y}_0 & & & & \\
		C^{x/y}_1 & B^{x/y}_1 & A^{x/y}_1 & & & \\
		& C^{x/y}_2 & B^{x/y}_2 & A^{x/y}_2  & & & \\
		& & C^{x/y}_3 & \ddots & \ddots & \\
		& & & \ddots & \ddots & \ddots \\
	\end{pmatrix}
\end{align*}
where the blocks themselves are diagonal for $J_x^{(a,b,c,d)}$ (sub-block-bandwidths $(0,0)$),
\begin{align*}
A^x_n &:= \begin{pmatrix}
		\alphaabcd_{n+1,0,1} & 0 & \hdots & 0 \\
		& \ddots & & \vdots & \\
		& & \alphaabcd_{n+1,n,1} & 0 \\
	    \end{pmatrix} \in \R^{(n+1)\times(n+2)}, \quad n = 0,1,2,\dots \\
B^x_n &:= \begin{pmatrix}
		\alphaabcd_{n,0,2} & & \\
		& \ddots & \\
		& & \alphaabcd_{n,n,2} \\
	    \end{pmatrix} \in \R^{(n+1)\times(n+1)} \quad n = 0,1,2,\dots \\
C^x_n &:= \big( A^x_n \big)^\top \in \R^{(n+1)\times n},  \quad n = 1,2,\dots \\ 
\nonumber
\end{align*}
and tridiagonal for $J_y^{(a,b,c,d)}$ (sub-block-bandwidths $(1,1)$),
\begin{align*}
A^y_n &:= \begin{pmatrix}
		\betaabcd_{n,0,8} & \betaabcd_{n,0,9} & & & \\
		\betaabcd_{n,1,7}& \ddots & \ddots & & \\
		& \ddots & \ddots & \ddots & \\
		& & \betaabcd_{n,n,7} & \betaabcd_{n,n,8} & \betaabcd_{n,n,9} \\
	    \end{pmatrix} \in \R^{(n+1)\times(n+2)}, \quad n = 0,1,2,\dots \\
B^y_n &:= \begin{pmatrix}
		\betaabcd_{n,0,5} & \betaabcd_{n,0,6} & & \\
		\betaabcd_{n,1,4} & \ddots & \ddots & \\
		& \ddots & \ddots & \betaabcd_{n,n-1,6} \\
		& & \betaabcd_{n,n,4} & \betaabcd_{n,n,5}
	    \end{pmatrix} \in \R^{(n+1)\times(n+1)}  \quad n = 0,1,2,\dots \\
C^y_n &:= \begin{pmatrix}
		\betaabcd_{n,0,2} & \betaabcd_{n,0,3} & & \\
		\betaabcd_{n,1,1} & \ddots & \ddots & \\
		& \ddots & \ddots & \betaabcd_{n,n-2,3} \\
		& & \ddots & \betaabcd_{n,n-1,2} \\
		& & & \betaabcd_{n,n,1} \\
	    \end{pmatrix} \in \R^{(n+1)\times n}, \quad n = 1,2,\dots
\end{align*}

Note that the sparsity of the Jacobi matrices (in particular the sparsity of the sub-blocks) comes from the natural sparsity of the three-term recurrences of the 1D OPs, meaning that the sparsity is not limited to the specific disk-slice case.

\subsection{Building the OPs} 

We can combine each system in (\ref{eqn:jacobimatricesdefinition}) into a block-tridiagonal system:
\begin{align*}
\renewcommand\arraystretch{1.3}
\begin{pmatrix}
		1 & & & \\
		B_0-G_0(x,y) & A_0 & & \\
		C_1 & B_1-G_1(x,y) & \quad A_1 \quad & \\
		& C_2 & B_2 - G_2(x,y)  & \ddots \\
		& & \ddots &\ddots
\end{pmatrix}
\bighdopabcd(x,y)
=
\begin{pmatrix}
	\hdopabcd_{0,0} \\ 0 \\ 0 \\ 0 \\ \vdots  \\
\end{pmatrix},
\end{align*}
where we note $\hdopabcd_{0,0}(x,y) \equiv \genjac_0^{(a,b,c+d+1)} \: \jac_0^{(d,c)}$, and for each $n = 0,1,2\dots$,
\begin{align*}
A_n &:= \begin{pmatrix}
		A^x_n \\
		A^y_n
	    \end{pmatrix} \in \R^{2(n+1)\times(n+2)}, \quad
C_n := \begin{pmatrix}
		C^x_n \\
		C^y_n
	    \end{pmatrix} \in \R^{2(n+1)\times n} \quad (n \ne 0), \nonumber \\
B_n &:= \begin{pmatrix}
		B^x_n \\
		B^y_n
	    \end{pmatrix} \in \R^{2(n+1)\times(n+1)}, \quad
G_n(x,y) := \begin{pmatrix}
		xI_{n+1} \\
		yI_{n+1}
	    \end{pmatrix} \in \R^{2(n+1)\times(n+1)}.
\end{align*}
 
For each $n = 0,1,2\dots$ let $\Dnt$ be any matrix that is a left inverse of $A_n$, i.e. such that $\Dnt A_n = I_{n+2}$. Multiplying our system by the preconditioner matrix that is given by the block diagonal matrix of the $\Dnt$'s, we obtain a lower triangular system \cite[p78]{dunkl2014orthogonal}, which can be expanded to obtain the recurrence:
\begin{align*}
\begin{cases}
\bighdopabcd_{-1}(x,y) := 0 \\
\bighdopabcd_{0}(x,y) := \hdopabcd_{0,0} \\
\bighdopabcd_{n+1}(x,y) = -\Dnt (B_n-G_n(x,y)) \bighdopabcd_n(x,y) - \Dnt C_n  \,\bighdopabcd_{n-1}(x,y), \quad n = 0,1,2,\dots.
\end{cases}
\end{align*}

Note that we can define an explicit \(\Dnt\) as follows:
\begin{align*}
\Dnt := \begin{pmatrix}
		\frac{1}{\alphaabcd_{n+1,0,1}} & &  \\
		& \ddots & & &  \\
		& & \ddots & & \\
		& & & \frac{1}{\alphaabcd_{n+1,n,1}} & \\
		\eta_{0} & 0 & \hdots & 0 & \eta_{1} & \hdots & \hdots & \eta_{n+1}
	    \end{pmatrix},
\end{align*}
 where
 \begin{align*}
	\eta_{n+1} &= \frac{1}{\betaabcd_{n,n,9}}, \\ 
	\eta_{n} &= -\frac{1}{\betaabcd_{n,n-1,9}} \big( \betaabcd_{n,n,8} \: \eta_{n+1} \big), \\
	\eta_j &= \frac{1}{\betaabcd_{n,j-1,9}} \big( \betaabcd_{n,n+j+1,7} \: \eta_{j+2} + \betaabcd_{n,n+j,8} \: \eta_{j+1} \big) \qfor j = n-1,n-2,\dots,1, \\
	\eta_0 &= \frac{1}{\alphaabcd_{n+1,0,1}} \big( \betaabcd_{n,1,7} \: \eta_{2} + \betaabcd_{n,0,8} \: \eta_{1} \big).
\end{align*}
It follows that we can apply $\Dnt$ in $O(n)$ complexity, and thereby calculate $\bighdopabcd_{0}(x,y)$  through $\bighdopabcd_{n}(x,y)$ in optimal $O(n^2)$ complexity.

For the disk-slice, $\betaabc_{n,k,2} = \betaabc_{n,k,5} = \betaabc_{n,k,8} \equiv 0$ for any $n, k$.

\section{Sparse partial differential operators}\label{Section:PDOs}

\begin{figure} 
\center
\begin{tikzpicture} 
\draw[black,solid,ultra thin] (0,0)--(0,4);
\draw[black,solid,ultra thin] (0,0)--(4,0);
\draw[black,solid,ultra thin] (4,0)--(4,4);
\draw[black,solid,ultra thin] (0,4)--(4,4);
\draw[black,fill=black] (0,0) circle (.5ex);
\draw[black,fill=black] (2,2) circle (.5ex);
\draw[black,fill=black] (4,4) circle (.5ex);
\draw[black,fill=black] (0,2) circle (.5ex);
\draw[black,fill=black] (0,4) circle (.5ex);
\draw[black,thick,->] (0.2,0.2)--(1.8,1.8);
\draw[black,thick,->] (2.2,2.2)--(3.8,3.8);
\draw[black,thick,->] (0,0.2)--(0,1.8);
\draw[black,thick,->] (0,2.2)--(0,3.8);
\draw [black,thick,->] (0.2,4.2) to [out=20,in=160] (3.8,4.2);
\draw[] (0,0) node[anchor=east] {$(0,0,0)$};
\draw[] (0,4) node[anchor=east] {$(0,0,2)$};
\draw[] (4,0) node[anchor=west] {$(2,2,0)$};
\draw[] (4,4) node[anchor=west] {$(2,2,2)$};
\draw[] (1,1) node[anchor=north] {$\tfrac{\partial}{\partial x}$};
\draw[] (3,3) node[anchor=north] {$\tfrac{\partial}{\partial x}$};
\draw[] (0,1) node[anchor=west] {$\tfrac{\partial}{\partial y}$};
\draw[] (0,3) node[anchor=west] {$\tfrac{\partial}{\partial y}$};
\draw[black,thick,->] (-2.5,3.4)--(-2.5,4);
\draw[black,thick,->] (-2.5,3.4)--(-1.9,3.4);
\draw[] (-2.3,3.2) node[anchor=west] {$(a,b)$};
\draw[] (-2.7,3.6) node[anchor=south] {$c$};
\draw[] (2,4.6) node[anchor=south] {basis conversion};
\end{tikzpicture} 
\caption{The Laplace operator acting on vectors of $\smash{\hdopnk^{(0,0,0)}}$ coefficients has a sparse matrix representation if the range is represented as vectors of $\smash{\hdopnk^{(2,2,2)}}$ coefficients. Here, the arrows indicate that the corresponding operation has a sparse matrix representation when the domain is $\smash{\hdopnkabc}$ coefficients, where $(a,b,c)$ is at the tail of the arrow, and the range is $\smash{\hdopnk^{(\tilde{a},\tilde{b},\tilde{c})}}$ coefficients, where $(\tilde{a},\tilde{b},\tilde{c})$ is at the head of the arrow. 
}
\label{fig:Laplace} 
\end{figure}
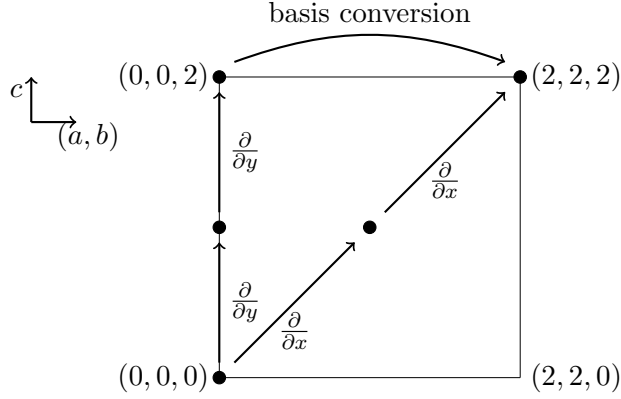

In this section, we concentrate on the disk-slice case, and simply note that similar arguments apply for the trapezium case. Recall that, for the disk-slice,
\begin{align*}
	\Omega := \{(x,y) \in \R^2 \quad | \quad \alpha < x < \beta, \: \gamma \rho(x) < y < \delta \rho(x)\}
\end{align*}
where
\begin{align*}
\begin{cases}
	(\alpha,\beta) &\subset (0,1) \\
	(\gamma,\delta) &:= (-1,1) \\
	\rho(x) &:= (1-x^2)^{\half}
\end{cases}.
\end{align*}
The 2D OPs on the disk-slice $\Omega$, orthogonal with respect to the weight
\begin{align*}
	\Wabc(x,y) &:= \genjacw^{(a,b,2c)}(x) \: \jacw^{(c)}\fpr(\frac{y}{\rho(x)}) \\
	&= (\beta - x)^a \: (x - \alpha)^b \: (1-x^2-y^2)^c, \quad (x,y) \in \Omega,
\end{align*}
are then given by:
\begin{align*}
	\hdopnkabc(x,y) := \genjacnmk^{(a, b, 2c+2k+1)}(x) \: \rho(x)^k \: \jac_k^{(c,c)}\fpr(\frac{y}{\rho(x)}), \quad (x,y) \in \Omega
\end{align*}
where the 1D OPs $\{\genjac^{(a, b, c)}_n\}$ are orthogonal on the interval $(\alpha, \beta)$ with respect to the weight
$$
\genjacw^{(a,b,c)}(x) :=  (\beta - x)^a \: (x - \alpha)^b \: \rho(x)^{c}
$$
and the 1D OPs $\{\jac^{(c, c)}_n\}$ are orthogonal on the interval $(\gamma, \delta) = (-1, 1)$ with respect to the weight
$$
\jacw^{(c)}(x) := (1-x)^c \: (1+x)^c = (1-x^2)^c.
$$

Denote the weighted OPs by
\begin{align*}
\bigWabc(x,y) := \Wabc(x,y) \: \bighdopabc(x,y),
\end{align*}
and recall that a function $f(x,y)$ defined on $\Omega$ is approximated by its expansion $f(x,y) = \bighdopabc(x,y)^\top \mathbf{f}$. 

\begin{definition}\label{def:differentialoperators}
Define the operator matrices $D_x^{(a,b,c)}, \: D_y^{(a,b,c)}, \: W_x^{(a,b,c)}, \: W_y^{(a,b,c)}$ according to:
\begin{align*}
{\partial f \over \partial x}&= \bighdop^{(a+1,b+1,c+1)}(x,y)^\top \: D_x^{(a,b,c)} \: \mathbf{f}, \\
{\partial f \over \partial y} &= \bighdop^{(a,b,c+1)}(x,y)^\top \: D_y^{(a,b,c)} \: \mathbf{f}, \\
{\partial \over \partial x}[\Wabc(x,y) \: f(x,y)] &= \bigW^{(a-1,b-1,c-1)}(x,y)^\top \: W_x^{(a,b,c)} \: \mathbf{f}, \\
{\partial \over \partial y}[\Wabc(x,y) \: f(x,y)] &= \bigW^{(a,b,c-1)}(x,y)^\top \: W_y^{(a,b,c)} \: \mathbf{f}.
\end{align*}
\end{definition}

The incrementing and decrementing of parameters as seen here is analogous to other well known orthogonal polynomial families' derivatives, for example the Jacobi polynomials on the interval, as seen in the DLMF \cite[(18.9.3)]{DLMF}, and on the triangle \cite{olver2018recurrence}.

\begin{theorem}\label{theorem:sparsityofdifferentialoperators}
The operator matrices $D_x^{(a,b,c)}, \: D_y^{(a,b,c)}, \: W_x^{(a,b,c)}, \: W_y^{(a,b,c)}$ from Definition \ref{def:differentialoperators} are sparse, with banded-block-banded structure. More specifically:
\begin{itemize}
	\item $D_x^{(a,b,c)}$ has  block-bandwidths $(-1,3)$, and sub-block-bandwidths $(0, 2)$.
  	\item $D_y^{(a,b,c)}$ has  block-bandwidths $(-1,1)$, and sub-block-bandwidths $(-1,1)$.
	\item $W_x^{(a,b,c)}$ has  block-bandwidths $(3,-1)$, and sub-block-bandwidths $(2, 0)$.
  	\item $W_y^{(a,b,c)}$ has  block-bandwidths $(1,-1)$, and sub-block-bandwidths $(1,-1)$.
\end{itemize}
\end{theorem}

\begin{proof}
First, note that:
\begin{align}
	\genjacw^{(a,b,c) \: \prime}(x) &= - a \: \genjacw^{(a-1,b,c)}(x) + b \: \genjacw^{(a,b-1,c)}(x) + c \: \rho(x) \: \rho'(x) \:\genjacw^{(a,b,c-2)}(x), \label{eqn:derivativeofweightgenjac} \\
	\jacw^{(c) \: \prime}(y) &= - 2c \: y \: \jacw^{(c-1)}(y), \label{eqn:derivativeofweightjac} \\
	\rho(x) \: \rho'(x) &= -x. \label{eqn:rhoderivative}
\end{align}

We proceed with the case for the operator $D_y^{(a,b,c)}$ for partial differentiation by $y$. Since $\{\hdop^{(a,b,c+1)}_{m,j}\}$ for $m = 0,\dots,n-1$, $j = 0,\dots,m$ is an orthogonal basis for any degree $n-1$ polynomial, we can expand $\pddy \hdopnkabc = \sum_{m=0}^{n-1} \sum_{j=0}^m c_{m,j}^y \: \hdop^{(a, b, c+1)}_{m,j}$. The coefficients of the expansion are then the entries of the relevant operator matrix. We can use an integration-by-parts argument to show that the only non-zero coefficient of this expansion is when $m = n-1$, $j = k-1$. First, note that
\begin{align*}
	c_{m,j}^y = {\ip<\pddy \hdopnkabc, \hdop^{(a, b, c+1)}_{m,j} >_{W^{(a, b, c+1)}}}{\norm{\hdop^{(a, b, c+1)}_{m,j}}^{-2}_{W^{(a, b+1)}}}.
\end{align*}
Then, using the change of variable $t = \frac{y}{\rho(x)}$, we have that
\begin{align*}
	&\ip<\pddy \hdopnkabc, \hdop^{(a, b, c+1)}_{m,j} >_{W^{(a, b, c+1)}} \\
	&= \iint_\Omega \: \Big[ \genjacnmk^{(a,b,2c+2k+1)}(x) \: \rho(x)^{k-1} \: \jac_k^{(c,c) \: \prime}\fpr(\frac{y}{\rho(x)}) \\
	& \quad \quad \quad \quad \quad \quad \cdot \: \genjacmmj^{(a,b,2c+2j+3)}(x) \: \rho(x)^{j} \: \jac_j^{(c+1,c+1)}\fpr(\frac{y}{\rho(x)}) \Big] \: \D y \: \D x \\
	&= \normgenjac^{(a, b, 2c+2k+1)} \: \ip< \genjacnmk^{(a,b,2c+2k+1)}, \rho(x)^{j-k+1} \: \genjacmmj^{(a,b,2c+2j+3)} >_{\genjacw^{(a,b,2c+2k+1)}} \\
	&\quad \quad \quad \quad \quad \quad \cdot \: \normjac^{(c+1)} \: \ip< \jac_k^{(c,c) \: \prime}, \: \jac_{j}^{(c+1,c+1)} >_{\jacw^{(c+1)}}
\end{align*}
Now, using (\ref{eqn:derivativeofweightjac}), integration-by-parts, and noting that the weight $\jacw^{(c)}$ is a polynomial of degree $2c$ and vanishes at the limits of the integral for positive parameter $c$, we have that
\begin{align*}
	\normjac^{(c+1)} \: \ip< \jac_k^{(c,c) \: \prime}, \: \jac_{j}^{(c+1,c+1)} >_{\jacw^{(c+1)}} &= \int_{\gamma}^\delta \: \jac_k^{(c,c) \: \prime}(y) \: \jac_{j}^{(c+1,c+1)}(y) \: \jacw^{(c+1)}(y) \: \D y \\
	&= - \int_{-1}^1 \: \jac_k^{(c,c)}(y) \: \ddy [\jacw^{(c+1)}(y) \: \jac_{j}^{(c+1,c+1)}(y)] \: \D y \\
	&= - \int_{-1}^1 \: \jac_k^{(c,c)} \: [ \jac_{j}^{(c+1,c+1) \: \prime} \: \jacw^{(c+1)} - 2 c \: y \: \jac_{j}^{(c+1,c+1)} \:\jacw^{(c)} ] \: \D y \\
	&= - \: \normjac^{(c)} \: \ip< \jac_k^{(c,c)}, \: \jacw^{(1)} \: \jac_{j}^{(c+1,c+1) \: \prime} -2c \:y \: \jac_{j}^{(c+1,c+1)} >_{\jacw^{(c)}}
\end{align*}
which is zero for $j < k-1$ by orthogonality. Further, when $j = k-1$, we have that
\begin{align*}
	&\normgenjac^{(a, b, 2c+2k+1)} \: \ip< \genjacnmk^{(a,b,2c+2k+1)}, \rho(x)^{j-k+1} \: \genjacmmj^{(a,b,2c+2j+3)} >_{\genjacw^{(a,b,2c+2k+1)}} \\
	&= \normgenjac^{(a, b, 2c+2k+1)} \: \ip< \genjacnmk^{(a,b,2c+2k+1)}, \genjacmmj^{(a,b,2c+2k+1)} >_{\genjacw^{(a,b,2c+2k+1)}} \\
	&= \normgenjac^{(a, b, 2c+2k+1)} \: \delta_{n,m+1},
\end{align*}
showing that the only possible non-zero coefficient is when $m=n-1, j=k-1$. Finally,
\begin{align*}
	c_{n-1,k-1}^y &= \ip< \jac_k^{(c,c) \: \prime}, \: \jac_{k-1}^{(c+1,c+1)} >_{\jacw^{(c+1)}}.
\end{align*}

We next consider the case for the operator $D_x^{(a,b,c)}$ for partial differentiation by $x$. Since $\{\hdop^{(a+1, b+1, c+1)}_{m,j}\}$ for $m = 0,\dots,n-1$, $j = 0,\dots,m$ is an orthogonal basis for any degree $n-1$ polynomial, we can expand $\pddx \hdopnkabc = \sum_{m=0}^{n-1} \sum_{j=0}^m c_{m,j}^x \: \hdop^{(a+1, b+1, c+1)}_{m,j}$. The coefficients of the expansion are then the entries of the relevant operator matrix. As before, we can use an integration-by-parts argument to show that the only non-zero coefficients of this expansion are when $m = n-1, n-2, n-3$, $j = k, k-1, k-2$ and $0 \le j \le m$. First, note that
\begin{align*}
	c_{m,j}^x = {\ip<\pddx \hdopnkabc, \hdop^{(a+1, b+1, c+1)}_{m,j} >_{W^{(a+1, b+1, c+1)}}}{\norm{\hdop^{(a+1, b+1, c+1)}_{m,j}}^{-2}_{W^{(a+1, b+1, c+1)}}}.
\end{align*}
Now, again using the change of variable $t= \frac{y}{\rho(x)}$, we have that
\begin{align}
	&\ip<\pddx \hdopnkabc, \hdop^{(a+1, b+1, c+1)}_{m,j} >_{W^{(a+1, b+1, c+1)}} \nonumber \\ 
	&= \Big( \int_\alpha^\beta \: \genjacnmk^{(a, b, 2c+2k+1) \: \prime} \: \genjacmmj^{(a+1, b+1, 2c+2j+3)} \: \rho^{k+j+1} \: \genjacw^{(a+1, b+1, 2c+2)} \: \D x \Big) \nonumber \\
	&\quad \quad \quad \quad \quad \cdot \: \Big( \int_\gamma^\delta \jac_k^{(c,c)} \: \jac_j^{(c+1,c+1)} \: \jacw^{(c+1)} \: \D t \Big) \nonumber \\ 
	& \quad + k \: \Big( \int_\alpha^\beta \: \genjacnmk^{(a, b, 2c+2k+1)} \: \genjacmmj^{(a+1, b+1, 2c+2j+3)} \: \rho^{k+j} \: \rho' \: \genjacw^{(a+1, b+1, 2c+2)} \: \D x \Big) \nonumber \\
	& \quad \quad \quad \quad \quad \cdot \: \Big( \int_\gamma^\delta \jac_k^{(c,c)} \: \jac_j^{(c+1,c+1)} \: \jacw^{(c+1)} \: \D t \Big) \nonumber \\
	& \quad - \Big( \int_\alpha^\beta \: \genjacnmk^{(a, b, 2c+2k+1)} \: \genjacmmj^{(a+1, b+1, 2c+2j+3)} \: \rho^{k+j} \: \rho' \: \genjacw^{(a+1, b+1, 2c+2)} \: \D x \Big) \nonumber \\
	& \quad \quad \quad \quad \quad \cdot \: \Big( \int_\gamma^\delta t \: \jac_k^{(c,c) \: \prime} \: \jac_j^{(c+1,c+1)} \: \jacw^{(c+1)} \: \D t \Big). \label{eqn:dxsparsity}
\end{align}
We will first show that the second factor of each term in (\ref{eqn:dxsparsity}) are zero for $j < k - 2$ and also for $j = k - 1$. To this end, observe that, for any integer $c$, $\jac^{(c,c)}(-t) = (-1)^k \: \jac^{(c,c)}(t)$ and so $\jac_k^{(c,c)}$ is an even polynomial for even $k$, and an odd polynomial for odd $k$. Thus, $\jac_k^{(c,c)} \: \jac_{k-1}^{(c+1,c+1)}$ is an odd polynomial for any $k$. Hence
\begin{align*}
	\int_\gamma^\delta \jac_k^{(c,c)} \: \jac_j^{(c+1,c+1)} \: \jacw^{(c+1)} \: \D t = \int_{-\delta}^\delta \jac_k^{(c,c)} \: \jac_j^{(c+1,c+1)} \: \jacw^{(1)} \: \jacw^{(c)} \: \D t 
\end{align*}
is zero for $j < k - 2$ by orthogonality, and is zero for $j = k-1$ due to symmetry over the domain. Moreover, $t \: \jac_k^{(c,c) \: \prime}(t) \: \jac_j^{(c+1,c+1)}(t)$ is also an odd polynomial for any $k$ and so
\begin{align*}
	\int_\gamma^\delta t \: \jac_k^{(c,c) \: \prime}(t) \: \jac_j^{(c+1,c+1)}(t) \: \jacw^{(c+1)}(t) \: \D t
\end{align*}
is zero for $j = k-1$ due to symmetry over the domain, and
\begin{align*}
	&\int_\gamma^\delta t \: \jac_k^{(c,c) \: \prime} \: \jac_j^{(c+1,c+1)} \: \jacw^{(c+1)} \: \D t \\
	&= - \int_{-\delta}^\delta \jac_k^{(c,c)} \: \frac{\D}{\D t} \big[ t \: \jac_j^{(c+1,c+1)} \: \jacw^{(c+1)} \big] \: \D t \\
	&= - \int_{-\delta}^\delta \jac_k^{(c,c)} \: \big[ \jac_j^{(c+1,c+1)} \: \jacw^{(c+1)} + t \: \jac_j^{(c+1,c+1) \: \prime} \: \jacw^{(c+1)} - 2 c \: t^2 \: \jac_j^{(c+1,c+1)} \: \jacw^{(c)} \big] \: \D t \\
	&= - \: \normjac^{(c)} \: \ip< \jac_k^{(c,c)}, \: \jac_j^{(c+1,c+1)} \: \jacw^{(1)} + t \: \jac_{j}^{(c+1,c+1) \: \prime} \: \jacw^{(1)} - 2c \: t^2 \: \jac_{j}^{(c+1,c+1)}  >_{\jacw^{(c)}}
\end{align*}
which is zero for $j < k - 2$ by orthogonality. Thus, (\ref{eqn:dxsparsity}) is zero for $j \notin \{k-2, k\}$.

Now, using (\ref{eqn:derivativeofweightgenjac}), integration-by-parts, and noting that the weight $\genjacw^{(a,b,2c)}$ is a polynomial degree $a+b+2c$ and vanishes at the limits of the integral for positive parameters $a,b,c$, we have that
\begin{align}
	&\int_\alpha^\beta \: \genjacnmk^{(a, b, 2c+2k+1) \: \prime} \: \genjacmmj^{(a+1, b+1, 2c+2j+3)} \: \rho^{k+j+1} \: \genjacw^{(a+1, b+1, 2c+2)} \: \D x \nonumber  \\
	&= \int_\alpha^\beta \: \genjacnmk^{(a, b, 2c+2k+1) \: \prime} \: \genjacmmj^{(a+1, b+1, 2c+2j+3)} \: \genjacw^{(a+1, b+1, 2c+k+j+3)} \: \D x \nonumber \\
	&= - \int_\alpha^\beta \: \genjacnmk^{(a, b, 2c+2k+1)} \: \ddx \Big[ \genjacmmj^{(a+1, b+1, 2c+2j+3)} \: \genjacw^{(a+1, b+1, 2c+k+j+3)} \Big] \: \D x \nonumber \\
	&= - \int_\alpha^\beta \: \genjacnmk^{(a, b, 2c+2k+1)} \: \Big\{ \genjacmmj^{(a+1, b+1, 2c+2j+3) \: \prime} \: \genjacw^{(a+1, b+1, 2c+k+j+3)} \nonumber \\
	&\quad \quad \quad \quad \quad \quad \quad \quad \quad \quad \quad \quad + (2c+k+j+3) \: \rho \: \rho' \:  \genjacw^{(a+1, b+1, 2c+k+j+1)} \: \genjacmmj^{(a+1, b+1, 2c+2j+3)} \nonumber \\
	&\quad \quad \quad \quad \quad \quad \quad \quad \quad \quad \quad \quad + (b+1) \: \genjacw^{(a+1, b, 2c+k+j+3)} \: \genjacmmj^{(a+1, b+1, 2c+2j+3)} \nonumber \\
	&\quad \quad \quad \quad \quad \quad \quad \quad \quad \quad \quad \quad - (a+1) \: \genjacw^{(a, b+1, 2c+k+j+3)} \: \genjacmmj^{(a+1, b+1, 2c+2j+3)} \Big\} \: \D x \nonumber \\
	&= - \: \normgenjac^{(a, b, 2c+2k+1)} \:  \Big\{ \ip<\genjacnmk^{(a, b, 2c+2k+1)}, \: \genjacw^{(1, 1, j-k+2)} \: \genjacmmj^{(a+1, b+1, 2c+2j+3) \: \prime}>_{\genjacw^{(a, b, 2c+2k+1)}} \nonumber \\
		&\quad \quad \quad + (2c+k+j+3) \: \ip<\genjacnmk^{(a, b, 2c+2k+1)}, \: \rho \: \rho' \: \genjacw^{(1,1, j-k)} \genjacmmj^{(a+1, b+1, 2c+2j+3)} >_{\genjacw^{(a, b, 2c+2k+1)}} \nonumber \\
		&\quad \quad \quad + (b+1) \: \ip<\genjacnmk^{(a, b, 2c+2k+1)}, \: \genjacw^{(1,0, j-k+2)} \genjacmmj^{(a+1, b+1, 2c+2j+3)} >_{\genjacw^{(a, b, 2c+2k+1)}} \nonumber \\
		&\quad \quad \quad - (a+1) \: \ip<\genjacnmk^{(a, b, 2c+2k+1)}, \: \genjacw^{(0,1, j-k+2)} \genjacmmj^{(a+1, b+1, 2c+2j+3)} >_{\genjacw^{(a, b, 2c+2k+1)}} \Big\}. \label{eqn:dxsparsityR}
\end{align}
By recalling (\ref{eqn:rhoderivative}) and noting that $j-k$ is even by the earlier argument, we can see $\rho \: \rho' \: \genjacw^{(1,1, j-k)}$, $\genjacw^{(1,0, j-k+2)}$ and $\genjacw^{(1,0, j-k+2)}$ are all polynomials, and further that 
$$\text{deg}(\rho \: \rho' \: \genjacw^{(1,1, j-k)}) = \text{deg} (\genjacw^{(1,0, j-k+2)}) = \text{deg} (\genjacw^{(0,1, j-k+2)}) = 3 + j-k.$$ 
Hence, by orthogonality, each term in (\ref{eqn:dxsparsityR}) is is zero for $m-j+3+j-k < n-k \iff m < n - 3$.

Finally,
\begin{align*}
	&\int_\alpha^\beta \: \genjacnmk^{(a, b, 2c+2k+1)} \: \genjacmmj^{(a+1, b+1, 2c+2j+3)} \: \rho^{k+j} \: \rho' \: \genjacw^{(a+1, b+1, 2c+2)} \: \D x \\
	&= \normgenjac^{(a, b, 2c+2k+1)} \: \ip< \genjacnmk^{(a, b, 2c+2k+1)}, \: \rho \: \rho' \: \genjacw^{(1, 1, j-k)} \: \genjacmmj^{(a+1, b+1, 2c+2j+3)} >_{\genjacw^{(a, b, 2c+2k+1)}}
\end{align*}
which is also zero for $m < n - 3$. Thus
\begin{align*}
	&\ip<\pddx \hdopnkabc, \hdop^{(a+1, b+1, c+1)}_{m,j} >_{W^{(a+1, b+1, c+1)}} = 0 \qfor m < n - 3, \: j \notin \{k-2, k\},
\end{align*}
showing that the only possible non-zero coefficients $c_{m,j}^x$ are when $m = n-3,\dots,n$ and $j = k-2,k$ $(j \le m)$.

We can gain the non-zero entries of the weighted differential operators similarly, by noting that for the disk-slice
\begin{align}
	\pddx \Wabc(x,y) &= -a W^{(a-1, b, c)}(x,y) + bW^{(a, b, c)}(x,y) + 2c \rho(x) \: \rho'(x) \: W^{(a, b, c-1)} \label{eqn:weightderivativex} \\
	\pddy \Wabc(x,y) &= -2c \: y \: W^{(a, b, c-1)}(x,y) \label{eqn:weightderivativey}
\end{align}
and also that
\begin{align*}
	\ip< \Wabc \hdopnkabc, W^{(\tilde{a},\tilde{b},\tilde{c})} \hdopmj^{(\tilde{a},\tilde{b},\tilde{c})} >_{W^{(-\tilde{a},-\tilde{b},-\tilde{c})}} = \ip< \hdopnkab, \hdopmj^{(\tilde{a},\tilde{b})} >_\Wabc.
\end{align*}

\end{proof}

There exist conversion matrix operators that increment/decrement the parameters, transforming the OPs from one (weighted or non-weighted) parameter space to another. 

\begin{definition}\label{def:parametertransformationoperators}
Define the operator matrices 
$$
T^{(a,b,c)\to(a+1,b+1,c)}, \quad T^{(a,b,c)\to(a,b,c+1)}\qqand T^{(a,b,c)\to(a+1,b+1,c+1)}
$$ 
for conversion between non-weighted spaces, and 
$$
T_W^{(a,b,c)\to(a-1,b-1,c)}, \quad T_W^{(a,b,c)\to(a,b,c-1)} \qqand T_W^{(a,b,c)\to(a-1,b-1,c-1)}
$$ 
for conversion between weighted spaces, according to:
\begin{align*}
	\bighdop^{(a,b,c)}(x,y) &= \Big(T^{(a,b,c)\to(a+1,b+1,c)} \Big)^\top \: \bighdop^{(a+1,b+1,c)}(x,y) \\
	\bighdop^{(a,b,c)}(x,y) &= \Big(T^{(a,b,c)\to(a,b,c+1)} \Big)^\top \: \bighdop^{(a,b,c+1)}(x,y) \\
	\bighdop^{(a,b,c)}(x,y) &= \Big(T^{(a,b,c)\to(a+1,b+1,c+1)} \Big)^\top \: \bighdop^{(a+1,b+1,c+1)}(x,y) \\
	\bigW^{(a,b,c)}(x,y) &= \Big(T_W^{(a,b,c)\to(a-1,b-1,c)} \Big)^\top \: \bigW^{(a-1,b-1,c)}(x,y) \\
	\bigW^{(a,b,c)}(x,y) &= \Big(T_W^{(a,b,c)\to(a,b,c-1)} \Big)^\top \: \bigW^{(a,b,c-1)}(x,y) \\
	\bigW^{(a,b,c)}(x,y) &= \Big(T_W^{(a,b,c)\to(a-1,b-1,c-1)} \Big)^\top \: \bigW^{(a-1,b-1,c-1)}(x,y).
\end{align*}
\end{definition}

\begin{lemma}\label{lemma:sparsityofparametertransformationoperators}
The operator matrices in Definition \ref{def:parametertransformationoperators} are sparse, with banded-block-banded structure. More specifically:
\begin{itemize}
	\item $T^{(a,b,c)\to(a+1,b+1,c)}$ has block-bandwidth $(0,2)$, with diagonal blocks.
  	\item $T^{(a,b,c)\to(a,b,c+1)}$ has block-bandwidth $(0,2)$ and sub-block-bandwidth $(0,2)$.
	\item $T^{(a,b,c)\to(a+1,b+1,c+1)}$ has block-bandwidth $(0,4)$ and sub-block-bandwidth $(0,2)$.
	\item $T_W^{(a,b,c)\to(a-1,b-1,c)}$ has block-bandwidth $(2,0)$ with diagonal blocks.
  	\item $T_W^{(a,b,c)\to(a,b,c-1)}$ has block-bandwidth $(2,0)$ and sub-block-bandwidth $(2,0)$.
	\item $T_W^{(a,b,c)\to(a-1,b-1,c-1)}$ has block-bandwidth $(4,0)$ and sub-block-bandwidth $(2,0)$.
\end{itemize}
\end{lemma}

\begin{proof}
We proceed with the case for the non-weighted operators $T^{(a,b)\to(a+\tilde{a},b+\tilde{b},c+\tilde{c})}$, where $\tilde{a}, \tilde{b}, \tilde{c} \in \{0,1\}$. Since $\{\hdop^{(a+\tilde{a},b+\tilde{b},c+\tilde{c})}_{m,j}\}$ for $m = 0,\dots,n$, $j = 0,\dots,m$ is an orthogonal basis for any degree $n$ polynomial, we can expand $\hdopnkabc = \sum_{m=0}^{n} \sum_{j=0}^m c_{m,j} \: \hdop^{(a+\tilde{a},b+\tilde{b},c+\tilde{c})}_{m,j}$. The coefficients of the expansion are then the entries of the relevant operator matrix. We will show that the only non-zero coefficients are for $m \ge n - \tilde{a} - \tilde{b} - 2\tilde{c} $, $j \ge k-2\tilde{c}$ and $0 \le j \le m$. First, note that
\begin{align*}
	c_{m,j} = {\ip< \hdopnkabc, \hdop^{(a+\tilde{a},b+\tilde{b},c+\tilde{c})}_{m,j} >_{W^{(a+\tilde{a},b+\tilde{b},c+\tilde{c})}}}{\norm{\hdop^{(a+\tilde{a},b+\tilde{b},c+\tilde{c})}_{m,j}}^{-2}_{W^{(a+\tilde{a},b+\tilde{b},c+\tilde{c})}}}.
\end{align*}
Then, using the change of variable $t = \frac{y}{\rho(x)}$, we have that
\begin{align*}
	&\ip< \hdopnkabc, \hdop^{(a+\tilde{a},b+\tilde{b},c+\tilde{c})}_{m,j} >_{W^{(a+\tilde{a},b+\tilde{b},c+\tilde{c})}} \\
	&\quad = \normgenjac^{(a+\tilde{a},b+\tilde{b},2c+2\tilde{c})} \: \ip< \genjacnmk^{(a,b,2c+2k+1)}, \rho(x)^{k+j+1} \: \genjacmmj^{(a+\tilde{a},b+\tilde{b},2c+2\tilde{c}+2j+1)} >_{\genjacw^{(a+\tilde{a},b+\tilde{b},2c+2\tilde{c})}} \\ 
	&\quad \quad \quad \cdot \: \normjac^{(c+\tilde{c})} \: \ip< \jac_k^{(c,c)}, \: \jac_{j}^{(c+\tilde{c},c+\tilde{c})} >_{\jacw^{(c+\tilde{c})}} \\
	&\quad= \normgenjac^{(a,b,2c+2k+1)} \: \ip< \genjacnmk^{(a,b,2c+2k+1)}, \genjacw^{(\tilde{a}, \tilde{b}, \: 2\tilde{c}+j-k)} \: \genjacmmj^{(a+\tilde{a},b+\tilde{b},2c+2\tilde{c}+2j+1)} >_{\genjacw^{(a,b,2c+2k+1)}} \\
	&\quad \quad \quad \cdot \: \normjac^{(c)} \: \ip< \jac_k^{(c,c)}, \: \jacw^{(\tilde{c})} \: \jac_{j}^{(c+\tilde{c},c+\tilde{c})} >_{\jacw^{(c)}}.
\end{align*}
Since $\jacw^{(\tilde{c})}$ is a polynomial degree $2\tilde{c}$, we have that the above is then zero for $j < k-2\tilde{c}$. Further, since $\genjacw^{(\tilde{a}, \tilde{b}, \: 2\tilde{c}+j-k)}$ is a polynomial of degree $\tilde{a}+\tilde{b}+2\tilde{c}+j-k$, we have that the above is zero for $m-j+\tilde{a}+\tilde{b}+2\tilde{c}+j-k < n-k \iff m < n - \tilde{a}-\tilde{b}-2\tilde{c}$.

The sparsity argument for the weighted parameter transformation operators follows similarly.
\end{proof}

\begin{figure}[t]
	\begin{subfigure}{0.32\textwidth}
	\includegraphics[scale=0.35]{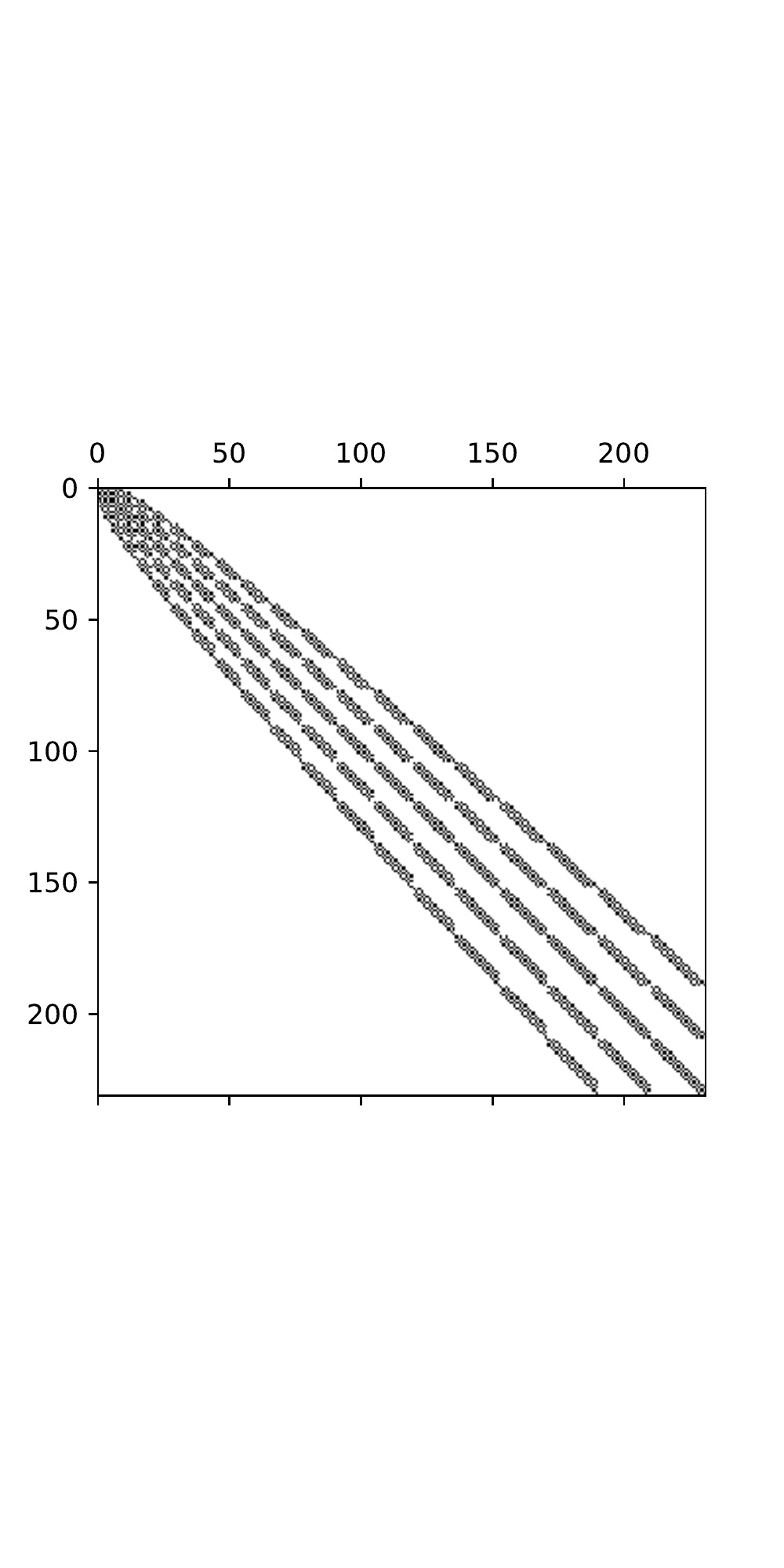}
        \centering
	\end{subfigure}
	\begin{subfigure}{0.32\textwidth}
	\includegraphics[scale=0.35]{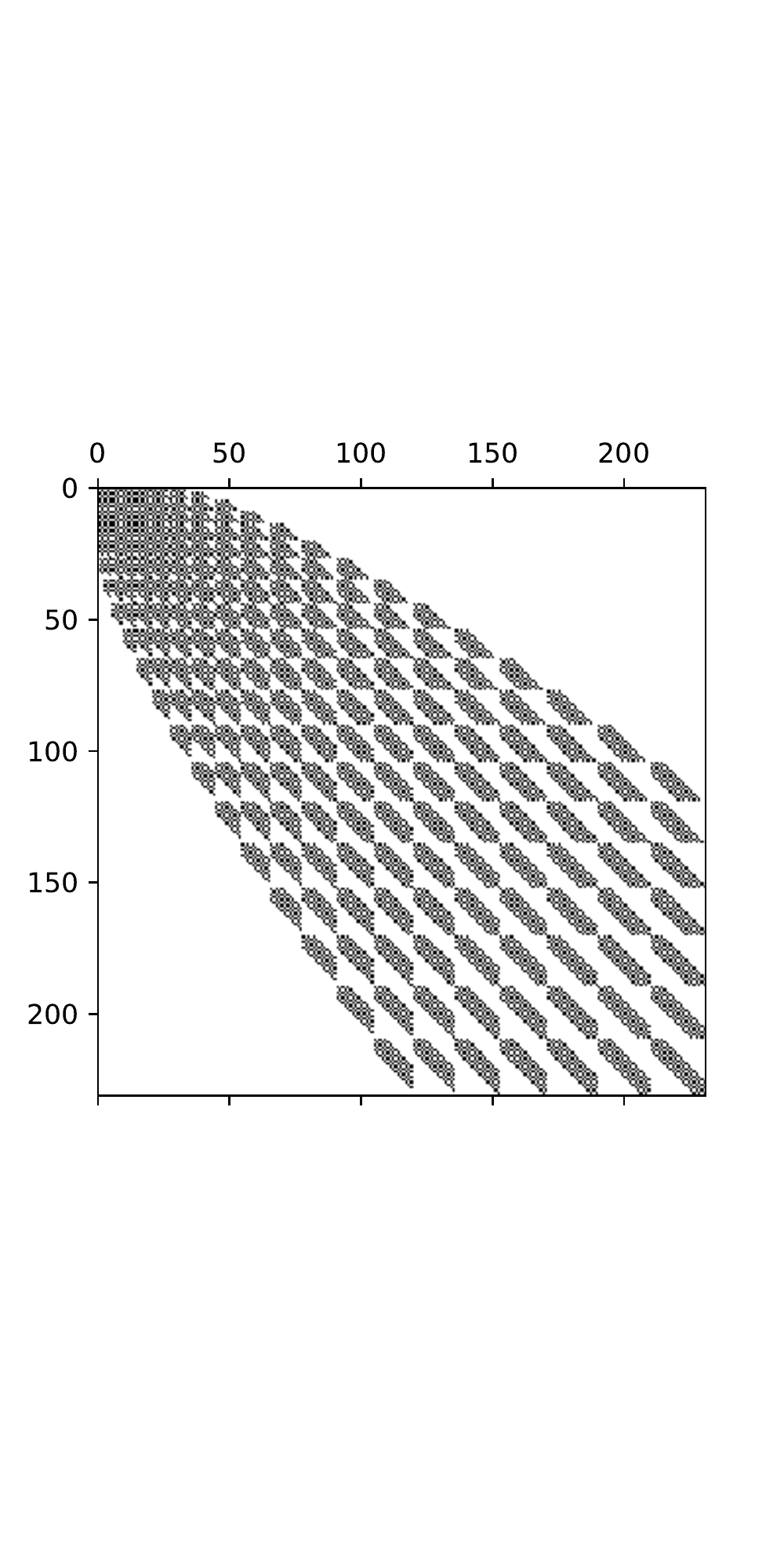}
        \centering
	\end{subfigure}
	\begin{subfigure}{0.32\textwidth}
	\includegraphics[scale=0.35]{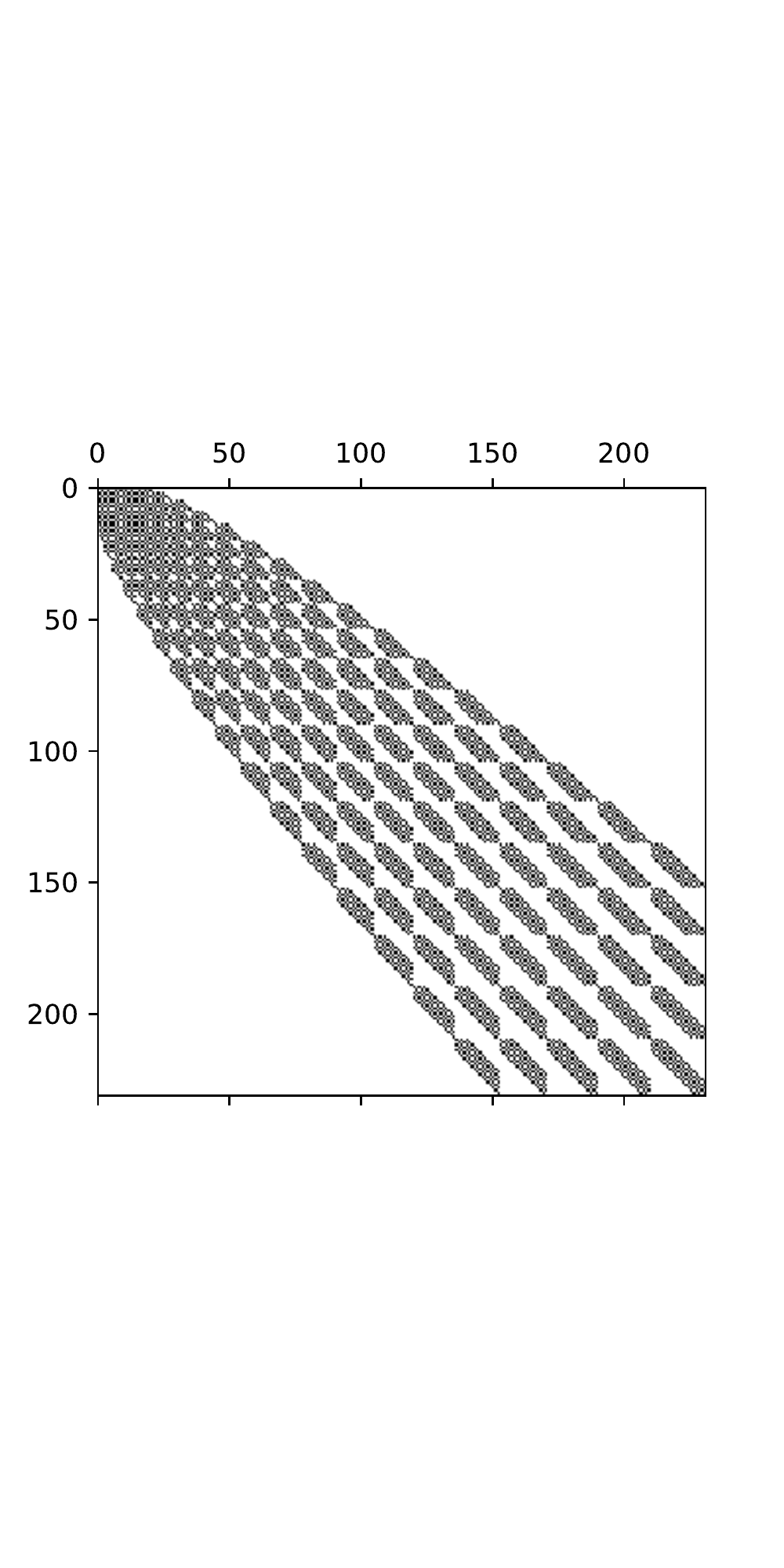}
        \centering
	\end{subfigure}
    	\caption{"Spy" plots of (differential) operator matrices, showing their sparsity. Left: the Laplace operator $\laplacewiii$. Centre: the weighted variable coefficient Helmholtz operator $\laplacewiii + k^2 \: T^{(0,0,0)\to(1,1,1)} \: V({J_x^{(0,0,0)}}^\top, {J_y^{(0,0,0)}}^\top) \: T_W^{(1,1,1)\to(0,0,0)}$ for $v(x,y) = 1 - (3(x-1)^2 + 5y^2)$ and $k = 200$. Right: the biharmonic operator $\biharmonictwo$.}
        \label{fig:sparsity}
        \centering
\end{figure}

General linear partial differential operators with polynomial variable coefficients can be constructed by composing the sparse representations for partial derivatives, conversion between bases, and Jacobi operators. As a canonical example, we can obtain the matrix operator for the Laplacian \(\Delta\), that will take us from coefficients for expansion in the weighted space
$$
\bigWiii(x,y) = \Wiii(x,y) \: \bighdopiii(x,y)
$$
to coefficients in the non-weighted space $\bighdopiii(x,y)$. Note that this construction will ensure the imposition of the Dirichlet zero boundary conditions on $\Omega$. The matrix operator for the Laplacian we denote $\laplacewiii$ acting on the coefficients vector is then given by
\begin{align*}
    \laplacewiii := D_x^{(0,0,0)} \: W_x^{(1,1,1)} + T^{(0,0,1)\to(1,1)} \: D_y^{(0,0,0)} \: T_W^{(1,1,0)\to(0,0,0)} \: W_y^{(1,1,1)}.
\end{align*}
Importantly, this operator will have banded-block-banded structure, and hence will be sparse, as seen in Figure \ref{fig:sparsity}.

Another important example is the Biharmonic operator $\Delta^2$, where we assume zero Dirichlet and Neumann conditions. To construct this operator, we first note that we can obtain the matrix operator for the Laplacian $\Delta$ that will take us from coefficients for expansion in the space $\bighdopooo(x,y)$ to coefficients in the space $\bighdop^{(2,2)}(x,y)$. We denote this matrix operator that acts on the coefficients vector as $\laplaceooo$, and is given by
\begin{align*}
    \laplaceooo := D_x^{(1,1,1)} \: D_x^{(0,0,0)} + T^{(1,1,2)\to(2,2,2)} \: D_y^{(1,1,1)} \: T^{(0,0,1)\to(1,1,1)} \: D_y^{(0,0,0)}.
\end{align*}
Further, we can represent the Laplacian as a map from coefficients in the space $\bigW^{(2,2)}$ to coefficients in the space $\bighdopooo$. Note that a function expanded in the $\bigW^{(2,2)}$ basis will satisfy both zero Dirichlet and Neumann boundary conditions on $\Omega$. We denote this matrix operator as $\laplacewttt$, and is given by
\begin{align*}
	\laplacewttt := W_x^{(1,1,1)} \: W_x^{(2,2,2)} + T_W^{(1,1,0)\to(0,0,0)} \: W_y^{(1,1,1)} \: T_W^{(2,2,1)\to(1,1,1)} \: W_y^{(2,2,2)}.
\end{align*}
We can then construct a matrix operator for $\Delta^2$ that will take coefficients in the space $\bigW^{(2,2,2)}$ to coefficients in the space $\bighdop^{(2,2,2)}$. Note that any function expanded in the $\bigW^{(2,2,2)}$ basis will satisfy both zero Dirichlet and zero Neumann boundary conditions on $\Omega$. The matrix operator for the Biharmonic operator is then given by
\begin{align*}
	\biharmonictwo = \laplaceooo \: \laplacewttt.
\end{align*}
The sparsity and structure of this biharmonic operator are seen in Figure \ref{fig:sparsity}.

\section{Computational aspects}\label{Section:Computation}

In this section we discuss how to take advantage of the proposed basis and sparsity structure in partial differential operators in practical computational applications.

\subsection{Constructing $\genjac_n^{(a,b,c)}(x)$}


It is possible to obtain the recurrence coefficients for the $\{\genjac_n^{(a,b,c)}\}$ OPs in (\ref{eqn:Hrecurrence}), by careful application of the Christoffel--Darboux formula \cite[18.2.12]{DLMF}. We explain the process here for the disk-slice case, however we note that a similar but simpler argument holds for the trapezium case. We thus first need to define a new set of `interim' 1D OPs.
\begin{definition}\label{def:InterimOPconstruction}
Let $\genjactw^{(a,b,c,d)}(x) := (\beta - x)^a \: (x - \alpha)^{b} \: (1-x)^{c} \: (1+x)^{d} $ be a weight function on the interval $(\alpha, \beta)$, and define the associated inner product by:
\begin{align}
	\ip< p, \: q >_{\genjactw^{(a,b,c,d)}} &:= \frac{1}{\normgenjact^{(a,b,c,d)}} \: \int_\alpha^\beta p(x) \: q(x) \: \genjactw^{(a,b,c,d)}(x) \: \D x \label{eqn:ipgenjact}
\end{align}
where
\begin{align}
	\normgenjact^{(a,b,c,d)} := \int_\alpha^\beta \: \genjactw^{(a,b,c,d)}(x) \: \D x
\end{align}
Denote the four-parameter family of orthonormal polynomials on $[\alpha,\beta]$ by $\{\genjact_n^{(a,b,c,d)}\}$, orthonormal with respect to the inner product defined in (\ref{eqn:ipgenjact}).
\end{definition}
Note that the OPs $\{\genjac_n^{(a,b,2c)}\}$ are then equivalent to the OPs $\{\genjact_n^{(a,b,c,c)}\}$. Let the recurrence coefficients for the OPs $\{\genjact_n^{(a,b,c,d)}\}$ be given by:
\begin{align}
	x \: \genjact_n^{(a,b,c,d)}(x) = \tilde{\beta}_n^{(a,b,c,d)} \: \genjact_{n+1}^{(a,b,c,d)}(x) + \tilde{\alpha}_n^{(a,b,c,d)} \:\genjact_n^{(a,b,c,d)}(x) + \tilde{\beta}_{n-1}^{(a,b,c,d)} \: \genjact_{n-1}^{(a,b,c,d)}(x)
\end{align}
\begin{proposition}
There exist constants $\mathcal{C}_n^{(a,b,c,d)}$, $\mathcal{D}_n^{(a,b,c,d)}$ such that
\begin{align}
	\genjact_n^{(a,b,c+1,d)}(x) &= \mathcal{C}_n^{(a,b,c,d)} \: \sum_{k=0}^n \: \genjact_{k}^{(a,b,c,d)}(1) \: \genjact_{k}^{(a,b,c,d)}(x) \label{eqn:genjacexpansioni} \\
	\genjact_n^{(a,b,c,d+1)}(x) &= \mathcal{D}_n^{(a,b,c,d)} \: \sum_{k=0}^n \: \genjact_{k}^{(a,b,c,d)}(-1) \: \genjact_{k}^{(a,b,c,d)}(x) \label{eqn:genjacexpansionii}
\end{align}
\end{proposition}

\begin{proof}
Fix $n,m \in \{0,1,\dots\}$ and without loss of generality, assume $m \le n$. First recall that
\begin{align*}
	\int_\alpha^\beta \: \genjact_n^{(a,b,c+1,d)}(x) \: \genjact_m^{(a,b,c+1,d)}(x) \: \genjactw^{(a,b,c+1,d)}(x) \: \D x &= \delta_{n,m} \: \normgenjact^{(a,b,c+1,d)}
\end{align*}
and define
\begin{align}
	\mathcal{C}_n^{(a,b,c,d)} &= \fpr(\frac{\normgenjact^{(a,b,c+1,d)}}{\normgenjact^{(a,b,c,d)} \: \genjact_n^{(a,b,c,d)}(1) \: \genjact_{n+1}^{(a,b,c,d)}(1) \: \tilde{\beta}_n^{(a,b,c,d)}})^\half, \label{eqn:genjactnormalisationi} \\
	\mathcal{D}_n^{(a,b,c,d)} &= (-1)^n \: \fpr(\frac{- \normgenjact^{(a,b,c,d+1)}}{\normgenjact^{(a,b,c,d)} \: \genjact_n^{(a,b,c,d)}(-1) \: \genjact_{n+1}^{(a,b,c,d)}(-1) \: \tilde{\beta}_n^{(a,b,c,d)}})^\half. \label{eqn:genjactnormalisationii}
\end{align}
Now, by the Christoffel--Darboux formula \cite[18.2.12]{DLMF}, we have that for any $x, y \in \R$,
\begin{align}
	\sum_{k=0}^n \: \genjact_{k}^{(a,b,c,d)}(y) \: \genjact_{k}^{(a,b,c,d)}(x) &= \tilde{\beta}_n^{(a,b,c,d)} \: \frac{\genjact_{n}^{(a,b,c,d)}(x) \: \genjact_{n+1}^{(a,b,c,d)}(y) - \genjact_{n+1}^{(a,b,c,d)}(x) \: \genjact_{n}^{(a,b,c,d)}(y)}{y - x}. \label{eqn:christoffeldarboux}
\end{align}
Then,
\begin{align*}
	&\int_\alpha^\beta \: \Big(\big[\mathcal{C}_n^{(a,b,c,d)} \: \sum_{k=0}^n \: \genjact_{k}^{(a,b,c,d)}(1) \: \genjact_{k}^{(a,b,c,d)}(x)\big] \\
	&\quad \quad \quad \cdot \big[\mathcal{C}_m^{(a,b,c,d)} \: \sum_{k=0}^m \: \genjact_{k}^{(a,b,c,d)}(1) \: \genjact_{k}^{(a,b,c,d)}(x)\big] \: \genjactw^{(a,b,c+1,d)}(x) \Big) \: \D x \\
	&= \mathcal{C}_n^{(a,b,c,d)} \: \mathcal{C}_m^{(a,b,c,d)} \: \tilde{\beta}_n^{(a,b,c,d)} \\
	&\quad \quad \cdot \sum_{k=0}^m \: \int_\alpha^\beta \: \Big(\genjact_{k}^{(a,b,c,d)}(1) \: \genjact_{k}^{(a,b,c,d)}(x) \\
	&\quad \quad \quad \quad \quad \quad \quad \cdot \big[\genjact_{n}^{(a,b,c,d)}(x) \: \genjact_{n+1}^{(a,b,c,d)}(1) - \genjact_{n+1}^{(a,b,c,d)}(x) \: \genjact_{n}^{(a,b,c,d)}(1)\big] \: \genjactw^{(a,b,c,d)}(x)\Big) \: \D x \\
	&= \delta_{m,n} \: {\mathcal{C}_n^{(a,b,c,d)}}^2 \: \tilde{\beta}_n^{(a,b,c,d)} \: \normgenjact^{(a,b,c,d)} \: \genjact_{n}^{(a,b,c,d)}(1) \: \genjact_{n+1}^{(a,b,c,d)}(1) \\
	&= \delta_{m,n} \: \normgenjact^{(a,b,c+1,d)}
\end{align*}
using (\ref{eqn:genjactnormalisationi}) and (\ref{eqn:christoffeldarboux}), showing that the RHS and LHS of (\ref{eqn:genjacexpansioni}) are equivalent. Further,
\begin{align*}
	&\int_\alpha^\beta \: \Big(\big[\mathcal{D}_n^{(a,b,c,d)} \: \sum_{k=0}^n \: \genjact_{k}^{(a,b,c,d)}(-1) \: \genjact_{k}^{(a,b,c,d)}(x)\big] \\
	&\quad \quad \quad \cdot \big[\mathcal{D}_m^{(a,b,c,d)} \: \sum_{k=0}^m \: \genjact_{k}^{(a,b,c,d)}(-1) \: \genjact_{k}^{(a,b,c,d)}(x)\big] \: \genjactw^{(a,b,c,d+1)}(x) \Big) \: \D x \\
	&= - \mathcal{D}_n^{(a,b,c,d)} \: \mathcal{D}_m^{(a,b,c,d)} \: \tilde{\beta}_n^{(a,b,c,d)} \\
	&\quad \quad \cdot \sum_{k=0}^m \: \int_\alpha^\beta \: \Big(\genjact_{k}^{(a,b,c,d)}(-1) \: \genjact_{k}^{(a,b,c,d)}(x) \\
	&\quad \quad \quad \quad \quad \quad \quad \cdot \big[\genjact_{n}^{(a,b,c,d)}(x) \: \genjact_{n+1}^{(a,b,c,d)}(-1) - \genjact_{n+1}^{(a,b,c,d)}(x) \: \genjact_{n}^{(a,b,c,d)}(-1)\big] \: \genjactw^{(a,b,c,d)}(x)\Big) \: \D x \\
	&= - \delta_{m,n} \: {\mathcal{D}_n^{(a,b,c,d)}}^2 \: \tilde{\beta}_n^{(a,b,c,d)} \: \normgenjact^{(a,b,c,d)} \: \genjact_{n}^{(a,b,c,d)}(-1) \: \genjact_{n+1}^{(a,b,c,d)}(-1) \\
	&= \delta_{m,n} \: \normgenjact^{(a,b,c,d+1)}
\end{align*}
using (\ref{eqn:genjactnormalisationii}) and (\ref{eqn:christoffeldarboux}), showing that the RHS and LHS of (\ref{eqn:genjacexpansionii}) are also equivalent. 
\end{proof}

\begin{proposition}
The recurrence coefficients for the OPs $\{\genjact_n^{(a,b,c+1,d)}\}$ are given by:
\begin{align}
	\tilde{\alpha}_n^{(a,b,c+1,d)} &= \frac{\genjact_{n+2}^{(a,b,c,d)}(1)}{\genjact_{n+1}^{(a,b,c,d)}(1)} \: \tilde{\beta}_{n+1}^{(a,b,c,d)} - \frac{\genjact_{n+1}^{(a,b,c,d)}(1)}{\genjact_{n}^{(a,b,c,d)}(1)} \: \tilde{\beta}_{n}^{(a,b,c,d)} + \tilde{\alpha}_{n+1}^{(a,b,c,d)}, \\
	\tilde{\beta}_n^{(a,b,c+1,d)} &= \frac{\mathcal{C}_n^{(a,b,c,d)}}{\mathcal{C}_{n+1}^{(a,b,c,d)}} \: \frac{\genjact_{n}^{(a,b,c,d)}(1)}{\genjact_{n+1}^{(a,b,c,d)}(1)} \: \tilde{\beta}_n^{(a,b,c,d)}.
\end{align}
The recurrence coefficients for the OPs $\{\genjact_n^{(a,b,c,d+1)}\}$ are given by:
\begin{align}
	\tilde{\alpha}_n^{(a,b,c,d+1)} &= \frac{\genjact_{n+2}^{(a,b,c,d)}(-1)}{\genjact_{n+1}^{(a,b,c,d)}(-1)} \: \tilde{\beta}_{n+1}^{(a,b,c,d)} - \frac{\genjact_{n+1}^{(a,b,c,d)}(-1)}{\genjact_{n}^{(a,b,c,d)}(-1)} \: \tilde{\beta}_{n}^{(a,b,c,d)} + \tilde{\alpha}_{n+1}^{(a,b,c,d)}, \\
	\tilde{\beta}_n^{(a,b,c,d+1)} &= \frac{\mathcal{D}_n^{(a,b,c,d)}}{\mathcal{D}_{n+1}^{(a,b,c,d)}} \: \frac{\genjact_{n}^{(a,b,c,d)}(-1)}{\genjact_{n+1}^{(a,b,c,d)}(-1)} \: \tilde{\beta}_n^{(a,b,c,d)}.
\end{align}
\end{proposition}

\begin{proof}
First, using (\ref{eqn:genjacexpansioni}) and (\ref{eqn:christoffeldarboux}) we have that
\begin{align}
	&(1-x) \: x \: \genjact_{n}^{(a,b,c+1,d)}(x) \nonumber \\
	&= \mathcal{C}_n^{(a,b,c,d)} \: \tilde{\beta}_{n}^{(a,b,c,d)} \: x \: \Big[\genjact_{n}^{(a,b,c,d)}(x) \: \genjact_{n+1}^{(a,b,c,d)}(1) - \genjact_{n+1}^{(a,b,c,d)}(x) \: \genjact_{n}^{(a,b,c,d)}(1)\Big] \nonumber \\
	&= \mathcal{C}_n^{(a,b,c,d)} \: \tilde{\beta}_{n}^{(a,b,c,d)} \nonumber \\
	& \quad \quad \cdot \Big[\Big(\tilde{\beta}_{n}^{(a,b,c,d)} \: \genjact_{n+1}^{(a,b,c,d)}(x) + \tilde{\alpha}_{n}^{(a,b,c,d)} \: \genjact_{n}^{(a,b,c,d)}(x) + \tilde{\beta}_{n-1}^{(a,b,c,d)} \: \genjact_{n-1}^{(a,b,c,d)}(x)\Big) \: \genjact_{n+1}^{(a,b,c,d)}(1) \nonumber \\
	& \quad \quad \quad - \Big(\tilde{\beta}_{n+1}^{(a,b,c,d)} \: \genjact_{n+2}^{(a,b,c,d)}(x) + \tilde{\alpha}_{n+1}^{(a,b,c,d)} \: \genjact_{n+1}^{(a,b,c,d)}(x) + \tilde{\beta}_{n}^{(a,b,c,d)} \: \genjact_{n}^{(a,b,c,d)}(x)\Big) \: \genjact_{n}^{(a,b,c,d)}(1) \Big] \label{eqn:genjactproofa}
\end{align}
Next, note that the recurrence coefficients for $\genjact_{n}^{(a,b,c+1,d)}(x)$ satisfy
\begin{align}
	&(1-x) \: x \: \genjact_{n}^{(a,b,c+1,d)}(x) \nonumber \\
	&= (1-x) \: \Big[\tilde{\beta}_{n}^{(a,b,c+1,d)} \: \genjact_{n+1}^{(a,b,c+1,d)}(x) + \tilde{\alpha}_{n}^{(a,b,c+1,d)} \: \genjact_{n}^{(a,b,c+1,d)}(x) + \tilde{\beta}_{n-1}^{(a,b,c+1,d)} \: \genjact_{n-1}^{(a,b,c+1,d)}(x)\Big] \nonumber \\
	&= \mathcal{C}_{n+1}^{(a,b,c,d)} \: \tilde{\beta}_{n}^{(a,b,c+1,d)} \: \tilde{\beta}_{n+1}^{(a,b,c,d)} \: \Big( \genjact_{n+1}^{(a,b,c,d)}(x) \: \genjact_{n+2}^{(a,b,c,d)}(1) - \genjact_{n+2}^{(a,b,c,d)}(x) \: \genjact_{n+1}^{(a,b,c,d)}(1) \Big) \nonumber \\
	& \quad \quad + \mathcal{C}_n^{(a,b,c,d)} \: \tilde{\alpha}_{n}^{(a,b,c+1,d)} \: \tilde{\beta}_{n}^{(a,b,c,d)} \: \Big( \genjact_{n}^{(a,b,c,d)}(x) \: \genjact_{n+1}^{(a,b,c,d)}(1) - \genjact_{n+1}^{(a,b,c,d)}(x) \: \genjact_{n}^{(a,b,c,d)}(1) \Big) \nonumber \\
	& \quad \quad + \mathcal{C}_{n-1}^{(a,b,c,d)} \: \tilde{\beta}_{n-1}^{(a,b,c+1,d)} \: \tilde{\beta}_{n-1}^{(a,b,c,d)} \: \Big( \genjact_{n-1}^{(a,b,c,d)}(x) \: \genjact_{n}^{(a,b,c,d)}(1) - \genjact_{n}^{(a,b,c,d)}(x) \: \genjact_{n-1}^{(a,b,c,d)}(1) \Big) \label{eqn:genjactproofb}
\end{align}
We can set $\tilde{\beta}_{-1}^{(a,b,c+1,d)} = 0$. By comparing coefficients of $\genjact_{n+2}^{(a,b,c,d)}(x)$ and $\genjact_{n+1}^{(a,b,c,d)}(x)$ in both (\ref{eqn:genjactproofa}) and (\ref{eqn:genjactproofb}) we obtain the desired recurrence coefficients for the OP $\genjact_{n}^{(a,b,c+1,d)}(x)$. The recurrence coefficients for the OPs $\genjact_{n}^{(a,b,c,d+1)}(x)$ are found similarly.
\end{proof}

\begin{corollary}
The recurrence coefficients for the OPs $\{\genjact_n^{(a,b,c+1,d)}\}$ can be written as:
\begin{align}
	\tilde{\alpha}_n^{(a,b,c+1,d)} &= \frac{\tilde{\beta}_{n-1}^{(a,b,c,d)}}{\chi_{n-1}^{(a,b,c,d)}(1)} - \frac{\tilde{\beta}_{n}^{(a,b,c,d)}}{\chi_{n}^{(a,b,c,d)}(1)} + \tilde{\alpha}_{n}^{(a,b,c,d)}, \\
	\tilde{\beta}_n^{(a,b,c+1,d)} &= \fpr(\frac{1 - \tilde{\alpha}_{n+1}^{(a,b,c,d)} - \frac{\tilde{\beta}_n^{(a,b,c,d)}}{\chi_n^{(a,b,c,d)}(1)}}{1 - \tilde{\alpha}_{n}^{(a,b,c,d)} - \frac{\tilde{\beta}_{n-1}^{(a,b,c,d)}}{\chi_{n-1}^{(a,b,c,d)}(1)}})^\half \tilde{\beta}_n^{(a,b,c,d)}.
\end{align}
The recurrence coefficients for the OPs $\{\genjact_n^{(a,b,c,d+1)}\}$ can be written as:
\begin{align}
	\tilde{\alpha}_n^{(a,b,c,d+1)} &= \frac{\tilde{\beta}_{n-1}^{(a,b,c,d)}}{\chi_{n-1}^{(a,b,c,d)}(-1)} - \frac{\tilde{\beta}_{n}^{(a,b,c,d)}}{\chi_{n}^{(a,b,c,d)}(-1)} + \tilde{\alpha}_{n}^{(a,b,c,d)}, \\
	\tilde{\beta}_n^{(a,b,c,d+1)} &= \fpr(\frac{- 1 + \tilde{\alpha}_{n+1}^{(a,b,c,d)} + \frac{\tilde{\beta}_n^{(a,b,c,d)}}{\chi_n^{(a,b,c,d)}(-1)}}{- 1 + \tilde{\alpha}_{n}^{(a,b,c,d)} + \frac{\tilde{\beta}_{n-1}^{(a,b,c,d)}}{\chi_{n-1}^{(a,b,c,d)}(-1)}})^\half \tilde{\beta}_n^{(a,b,c,d)}.
\end{align}
where 
\begin{align}
	\chi_{n}^{(a,b,c,d)}(y) &:= \frac{\genjact_{n+1}^{(a,b,c,d)}(y)}{\genjact_{n}^{(a,b,c,d)}(y)} \\
	&= \frac{1}{\tilde{\beta}_{n}^{(a,b,c,d)}} \: \fpr(y - \tilde{\alpha}_{n}^{(a,b,c,d)} - \frac{\tilde{\beta}_{n-1}^{(a,b,c,d)}}{\chi_{n-1}^{(a,b,c,d)}(y)}), \quad y \in \{-1,1\}.
\end{align}
\end{corollary}

These two propositions allow us to recursively obtain the recurrence coefficients for the OPs $\{\genjac_{n-k}^{(a,b,2c+2k + 1)}\}$ as $k$ increases to be large. 

\noindent{\bf Remark}:
The Corollary demonstrates that in order to obtain the recurrence coefficients $\{\alpha_{m}^{(a,b,2c+2k+1)}\}$, $\{\beta_{m}^{(a,b,2c+2k+1)}\}$ for some $m$ and $k$, we require that we obtain the recurrence coefficients $\{\alpha_{m+2}^{(a,b,2c+2(k-1)+1)}\}$, $\{\beta_{m+2}^{(a,b,2c+2(k-1)+1)}\}$. Thus, for large $N$, this recursive method of obtaining the recurrence coefficients requires a large initialisation (i.e. using the Lanczos algorithm to compute the recurrence coefficients $\{\alpha_{n}^{(a,b,2c+1)}\}$, $\{\beta_{n}^{(a,b,2c+1)}\}$ -- however, we only need to compute these once, and can store and save this initialisation to disk once computed, for the given values of $a, b, c$). 

\subsection{Quadrature rule on the disk-slice}

In this section we construct a quadrature rule exact for polynomials in the disk-slice $\Omega$ that can be used to expand functions in $\hdopnkabc(x,y)$ when $\Omega$ is a disk-slice.

\begin{theorem}
Denote the  Gauss quadrature nodes and weight on $[\alpha,\beta]$ with weight $(\beta - s)^a \: (s - \alpha)^b \: \rho(s)^{2c+1}$ as $(s_k,w_k^{(s)})$ , and
 on \([-1,1]\) with weight \((1-t^2)^c\) as $(t_k,w_k^{(t)})$. Define
\begin{align*}
	x_{i+(j-1)N} &:= s_j, \quad i,j = 1,\dots, \ceil{\frac{N+1}{2}}, \\
	y_{i+(j-1)N} &:= \rho(s_j) \: t_i, \quad i,j = 1,\dots, \ceil{\frac{N+1}{2}}, \\
	w_{i+(j-1)N} &:= w_j^{(s)} w_i^{(t)}, \quad  i,j = 1,\dots, \ceil{\frac{N+1}{2}}.
\end{align*}
Let $f(x,y)$ be a polynomial on $\Omega$. The quadrature rule is then
$$
\iint_\Omega f(x,y) \: \Wab(x,y) \: \D A \approx \half \sum_{j=1}^{M} w_j \: \big[ f(x_j, y_j) + f(x_j, -y_j) \big],
$$
where $M = \ceil{\half(N+1)}^2$, and the quadrature rule is exact if $f(x,y)$ is a polynomial of degree $\le N$.
\end{theorem}

\begin{proof}
We will use the substitution that
\begin{align*}
	x &= s, \quad y = \rho(s) \: t.
\end{align*}
First, note that, for $(x,y) \in \Omega$,
\begin{align*}
	\Wabc(x,y) &= \genjacw^{(a,b,2c)}(x) \: \jacw^{(c)}\fpr(\frac{y}{\rho(x)}) \\
	&= \genjacw^{(a,b,c2)}(s) \: \jacw^{(c)}(t) \\
	&=: V^{(a,b,c)}(s,t), \quad \text{for } (s,t) \in [\alpha,\beta] \times [-1,1].
\end{align*}

Let $f : \Omega \to \R$. Define the functions $f_e, f_o : \Omega \to \R$ by 
\begin{align*}
	f_e(x,y) &:= \half \Big(f(x, y) + f(x, -y)\Big), \quad \forall (x,y) \in \Omega\\
	f_o(x,y) &:= \half \Big(f(x, y) - f(x, -y)\Big), \quad \forall (x,y) \in \Omega
\end{align*}
so that $y \mapsto f_e(x,y)$ for fixed $x$ is an even function, and $y \mapsto f_o(x,y)$ for fixed $x$ is an odd function. Note that if $f$ is a polynomial, then $f_e(s, \rho(s)t)$ is a polynomial in $s \in [\alpha,\beta]$ for fixed $t$. 

Now, we have that
\begin{align*}
	\iint_\Omega f_e(x,y) \: \Wabc(x,y) \: \D y \: \D x &= \int_\alpha^\beta \int_{-1}^1 f_e\big(s,\rho(s) t\big) \: V^{(a,b,c)}(s,t) \: \rho(s) \: \D t \: \D s \\
	&= \int_\alpha^\beta \genjacw^{(a,b,2c+1)}(s) \: \Big( \int_{-1}^1 f_e\big(s,\rho(s) t\big) \: \jacw^{(c)}(t) \: \D t \Big) \: \D s \\
	&\approx \int_\alpha^\beta \genjacw^{(a,b,2c+1)}(s) \: \sum_{i=1}^{M_2} \Big( w_i^{(t)} f_e\big(s,\rho(s) t_i\big) \Big) \: \D s \quad (\star) \\
	&\approx \sum_{j=1}^{M_1} \Bigg( w_j^{(s)} \: \sum_{i=1}^{M_2} \Big( w_i^{(t)} f_e\big(s_j,\rho(s_j) t_i\big) \Big) \Bigg) \quad (\star \star) \\
	&= \sum_{k=1}^{M_1 M_2}  w_k \: f_e(x_k, y_k).
\end{align*}
Suppose $f$ is a polynomial in $x$ and $y$ of degree $N$, and hence that $f_e$ is a degree $\le N$ polynomial. First, note that the degree of the polynomial given by $x \mapsto f_e(x,y)$ for fixed $y$ is $\le N$ and the degree of the polynomial given by $y \mapsto f_e(x,y)$ for fixed $x$ is $\le N$. Also note that $s \mapsto f_e\big(s,\rho(s) t\big)$ for fixed $t$ is then a degree $N$ polynomial (since $\rho$ is a degree $1$ polynomial). Hence, we achieve equality at $(\star)$ if $2M_2 - 1 \ge N$ and we achieve equality at $(\star \star)$ if also $2M_1 - 1 \ge N$.

Next, note that 
\begin{align*}
	\iint_\Omega f_o(x,y) \: \Wabc(x,y) \: \D y \: \D x &= \int_\alpha^\beta \int_{-1}^1 f_o\big(s,\rho(s) t\big) \: V^{(a,b,c)}(s,t) \: \rho(s) \: \D t \: \D s \\
	&= \int_\alpha^\beta  \genjacw^{(a,b,2c+1)}(s) \: \Big( \int_{-1}^1 f_o\big(s,\rho(s) t\big) \: \jacw^{(c)}(t) \: \D t \Big) \: \D s \quad (\dagger) \\
	&= 0
\end{align*}
since the inner integral at $(\dagger)$ over $t$ is zero, due to the symmetry over the domain.

Hence, for a polynomial $f$ in $x$ and $y$ of degree $N$,
\begin{align*}
	\iint_\Omega f(x,y) \: \Wabc(x,y) \: \D y \: \D x &= \iint_\Omega \Big(f_e(x,y) + f_o(x,y)\Big) \: \Wabc(x,y) \: \D y \: \D x \\
	&= \iint_\Omega f_e(x,y) \: \Wabc(x,y) \: \D y \: \D x \\
	&= \sum_{j=1}^{M}  w_j \: f_e(x_j, y_j),
\end{align*}
where $M = \ceil{\half(N+1)}^2$.
\end{proof}

\subsection{Obtaining the coefficients for expansion of a function on the disk-slice}

Fix \(a,b,c \in \R\). Then for any function \(f : \Omega \to \R\) we can express \(f\) by
\begin{align*}
	f(x,y) \approx \sum_{n=0}^N \bighdop_n^{(a,b,c)}(x,y)^\top \: \mathbf{f}_n
\end{align*}
for N sufficiently large, where
\begin{align*}
	\bighdop^{(a,b,c)}_n(x,y) &:= \begin{pmatrix}
		\hdop^{(a,b,c)}_{n,0}(x,y) \\
		\vdots \\
		\hdop^{(a,b,c)}_{n,n}(x,y)
	\end{pmatrix} \in \R^{n+1} \quad \forall n = 0,1,2,\dots,N,
\end{align*}
and where
\begin{align*}
	\mathbf{f}_n &:= \begin{pmatrix}
		f_{n,0} \\
		\vdots \\
		f_{n,n}
	\end{pmatrix} \in \R^{n+1} \quad \forall n = 0,1,2,\dots,N, \quad
	f_{n,k} := \frac{\ip< f, \: \hdopnk^{(a,b,c)}>_{\Wabc}}{\norm{\hdopnk^{(a,b,c)}}_{\Wabc}}
\end{align*}
Recall from (\ref{eqn:normhdop}) that $\norm{\hdopnkabc}^2_\Wabc = \normgenjac^{(a,b,2c+2k+1)} \: \normjac^{(c)}$. Using the quadrature rule detailed in Section 4.2 for the inner product, we can calculate the coefficients $f_{n,k}$ for each $n = 0,\dots,N$, $k = 0,\dots,n$: 
\begin{align*}
	f_{n,k} &= \frac{1}{2 \: \normgenjac^{(a,b,2c+2k+1)} \: \normjac^{(c)}} \: \sum_{j=1}^{M} w_j \: \big[ f(x_j, y_j) \: \hdopnkabc(x_j, y_j) +f(x_j, -y_j) \: \hdopnkabc(x_j, -y_j) \big]
\end{align*}
where $M = \ceil{\half(N+1)}^2$.

\subsection{Calculating non-zero entries of the operator matrices}\label{subsection:Computation-operatormatrices}

The proofs of Theorem \ref{theorem:sparsityofdifferentialoperators} and Lemma \ref{lemma:sparsityofparametertransformationoperators} provide a way to calculate the non-zero entries of the operator matrices given in Definition \ref{def:differentialoperators} and Definition \ref{def:parametertransformationoperators}. We can simply use quadrature to calculate the 1D inner products, which has a complexity of $\bigO(N^3)$. This proves much cheaper computationally than using the 2D quadrature rule to calculate the 2D inner products, which has a complexity of $\bigO(N^4)$.

\section{Examples on the disk-slice with zero Dirichlet conditions}\label{Section:Examples}

We now demonstrate how the sparse linear systems constructed as above can be used to efficiently solve PDEs with zero Dirichlet conditions. We consider Poisson, inhomogeneous variable coefficient Helmholtz equation and the Biharmonic equation, demonstrating the versatility of the approach. 

\subsection{Poisson}

\begin{figure}[t]
	\begin{subfigure}{0.3\textwidth}
	\includegraphics[scale=0.3]{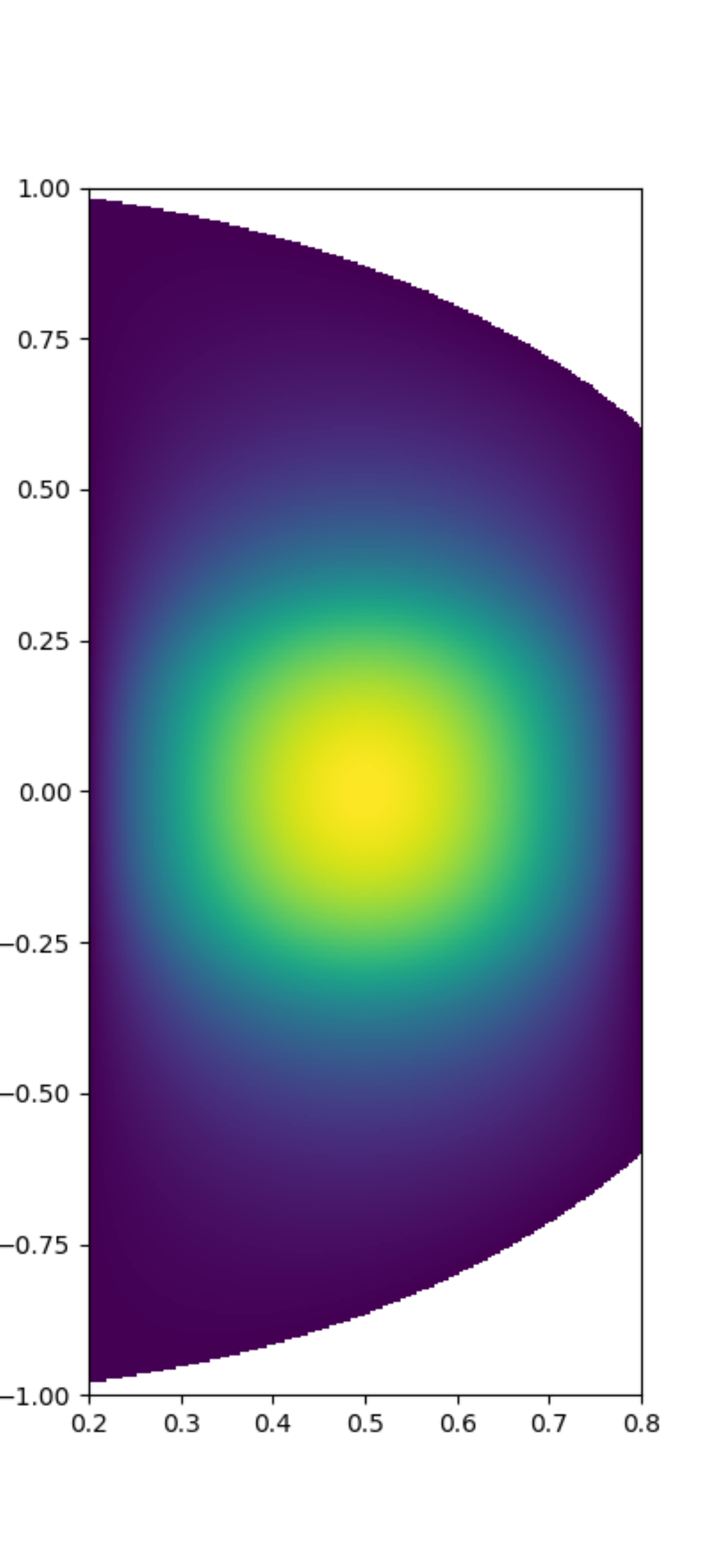}
	\centering
	\end{subfigure}
	\begin{subfigure}{0.5\textwidth}
	\centering
	\includegraphics[scale=0.5]{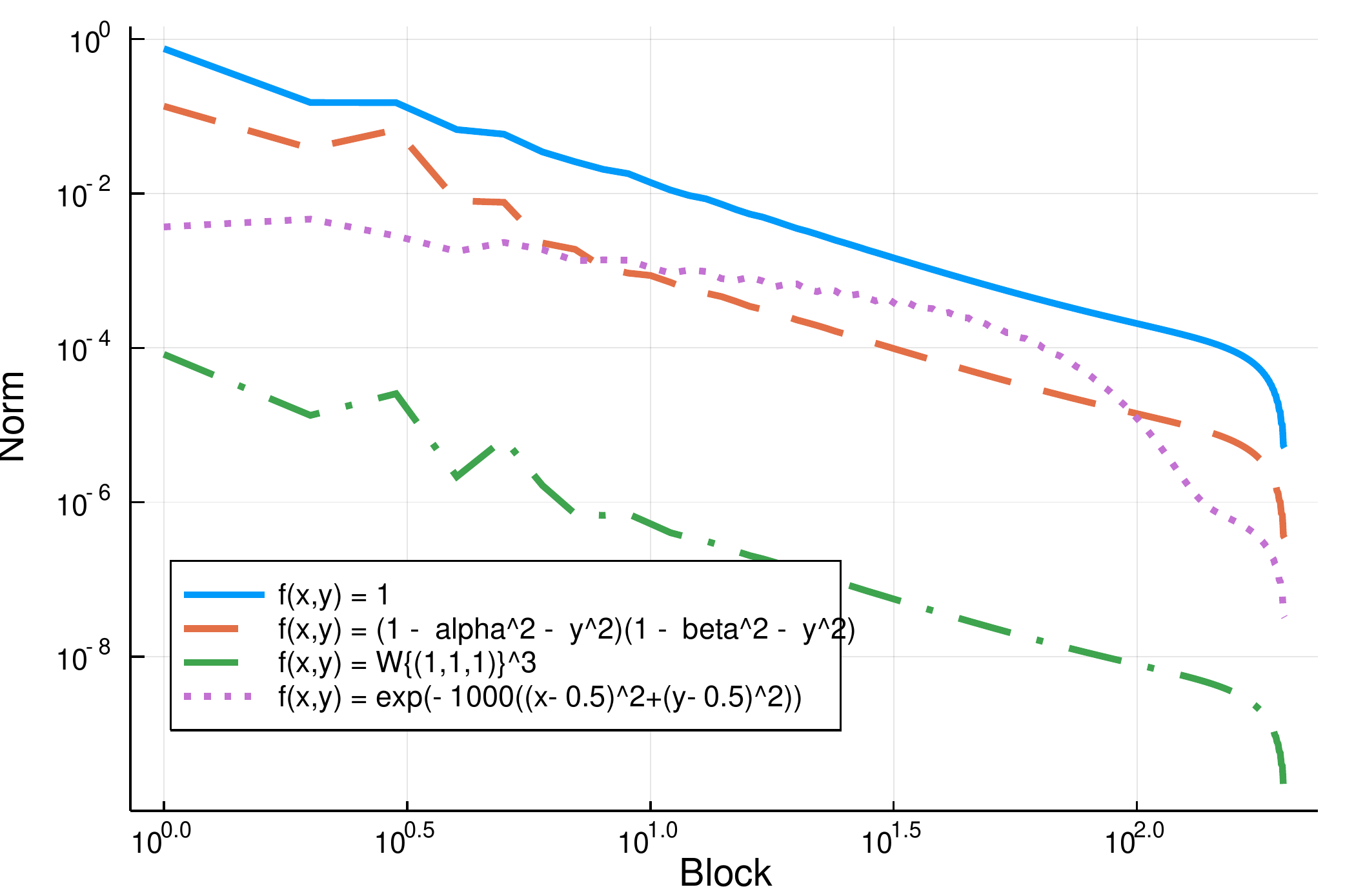}
	\end{subfigure}
	\caption{Left: The computed solution to $\Delta u = f$ with zero boundary conditions with $f(x,y) = 1 + \text{erf}(5(1 - 10((x - 0.5)^2 + y^2)))$. Right: The norms of each block of the computed solution of the Poisson equation with the given right hand side functions. This demonstrates algebraic convergence with the rate dictated by the decay at the corners, with spectral convergence observed when the right-hand side vanishes to all orders.}
	\centering
	\label{fig:poisson}
\end{figure}

\begin{figure}[t]
	\begin{subfigure}{0.3\textwidth}
	\includegraphics[scale=0.3]{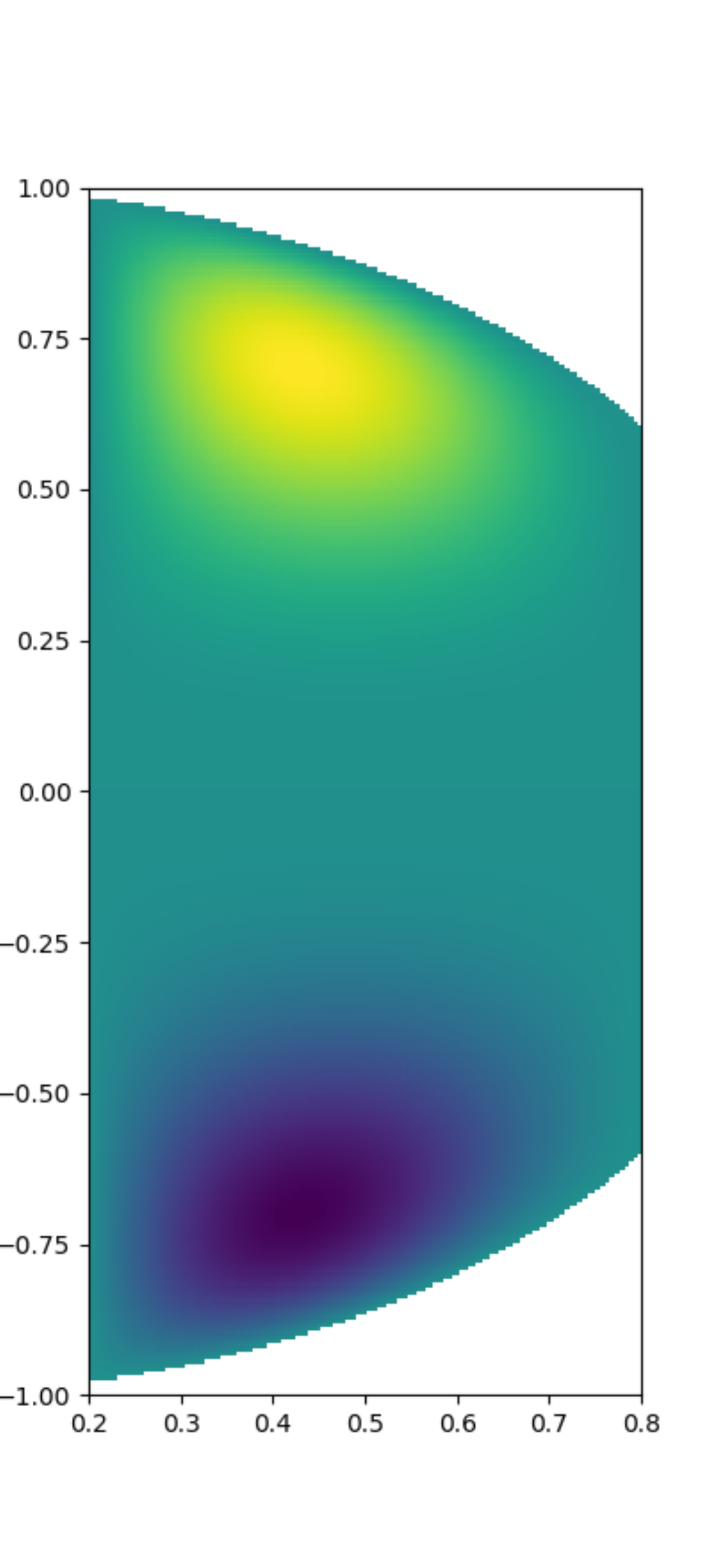}
	\centering
	\end{subfigure}
	\begin{subfigure}{0.3\textwidth}
	\centering
	\includegraphics[scale=0.3]{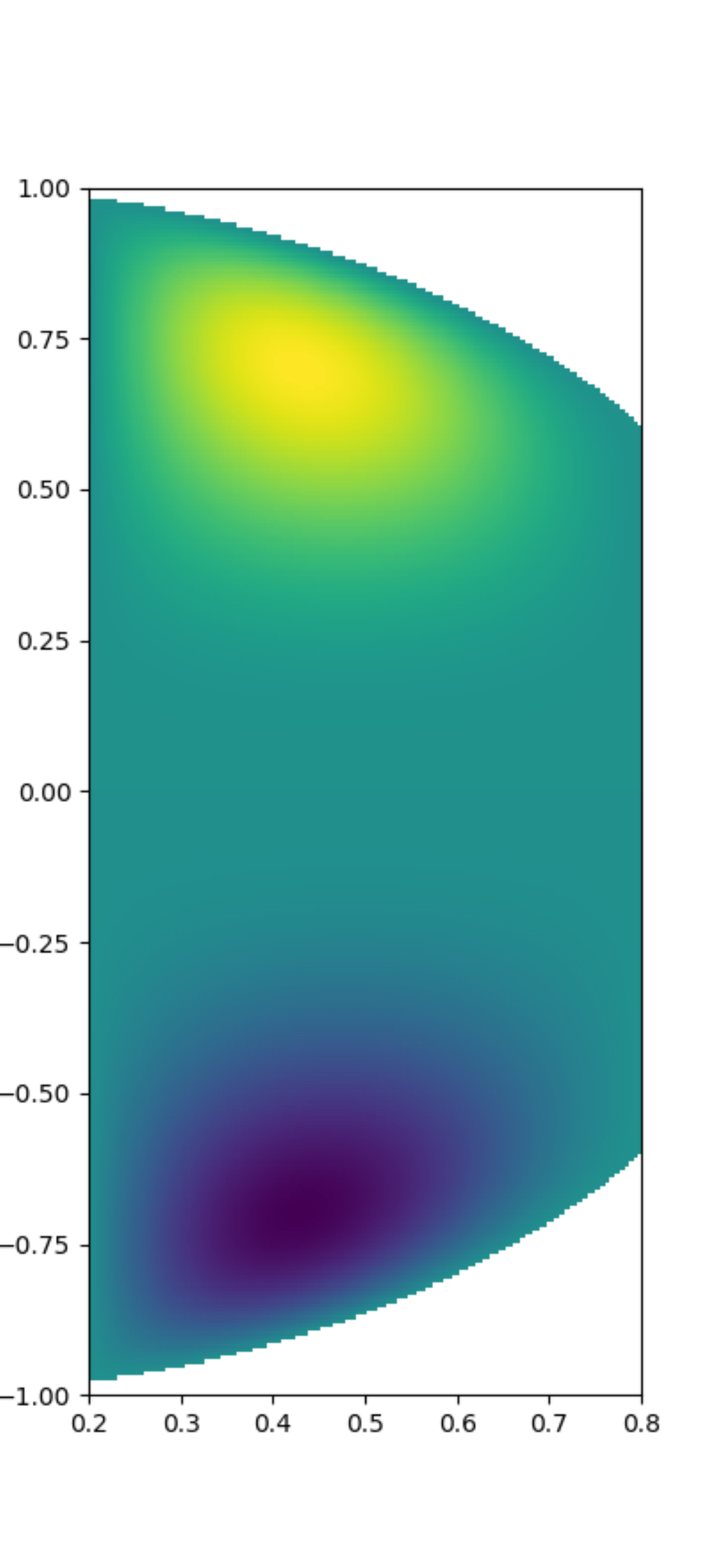}
	\centering
	\end{subfigure}
	\begin{subfigure}{0.3\textwidth}
	\includegraphics[scale=0.3]{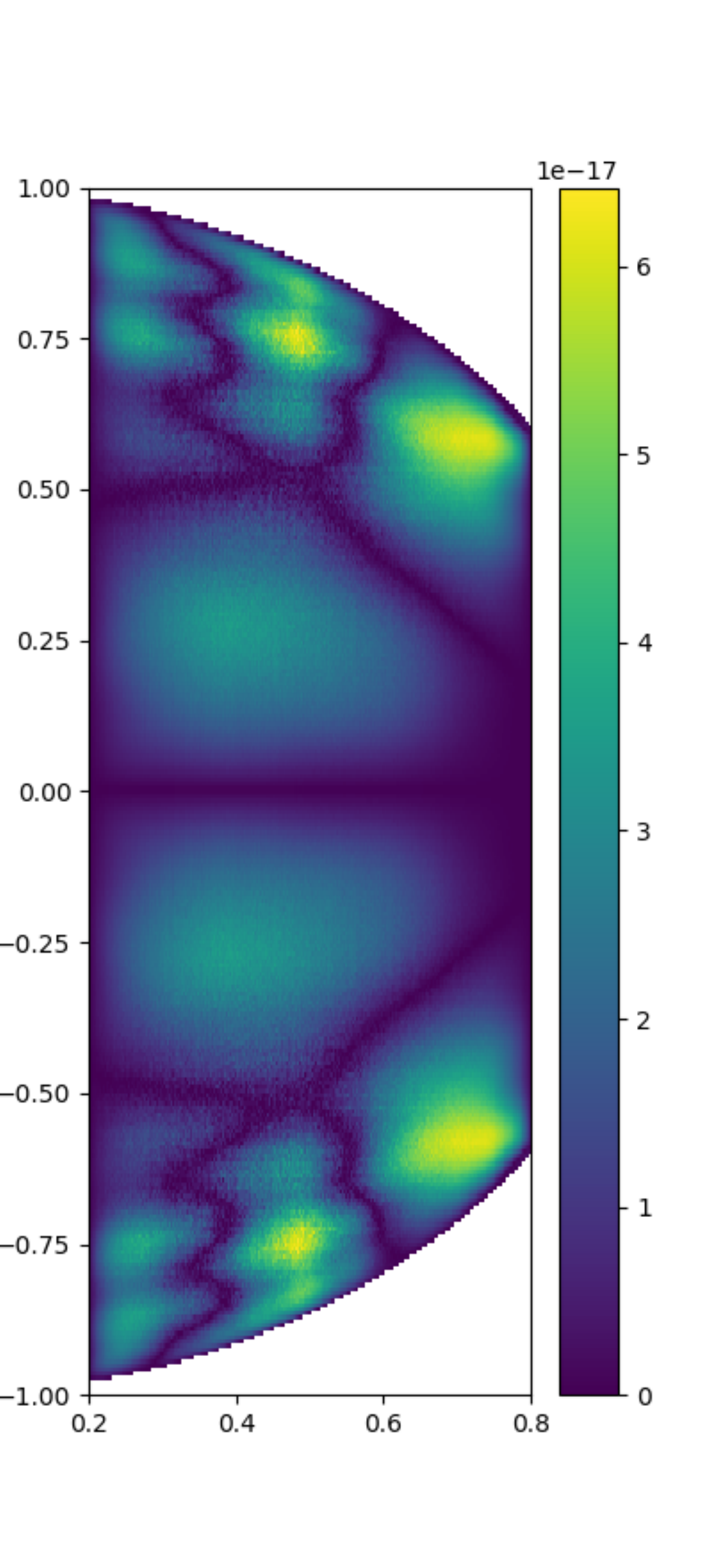}
	\centering
	\end{subfigure}
	\centering
	\caption{The computed solution to $\Delta u = f$ with zero boundary conditions compared with the exact solution $u(x,y) = \Wiii (x,y) y^3 \exp(x)$. Left: Computed. Centre: Exact. Right: Plot of the error (colourbar is shown to demonstrate magnitude of the error is of the order $10^{-17}$)}
	\centering
	\label{fig:poissonexact}
\end{figure}

\begin{figure}[t]
	\begin{subfigure}{0.3\textwidth}
	\centering
	\includegraphics[scale=0.3]{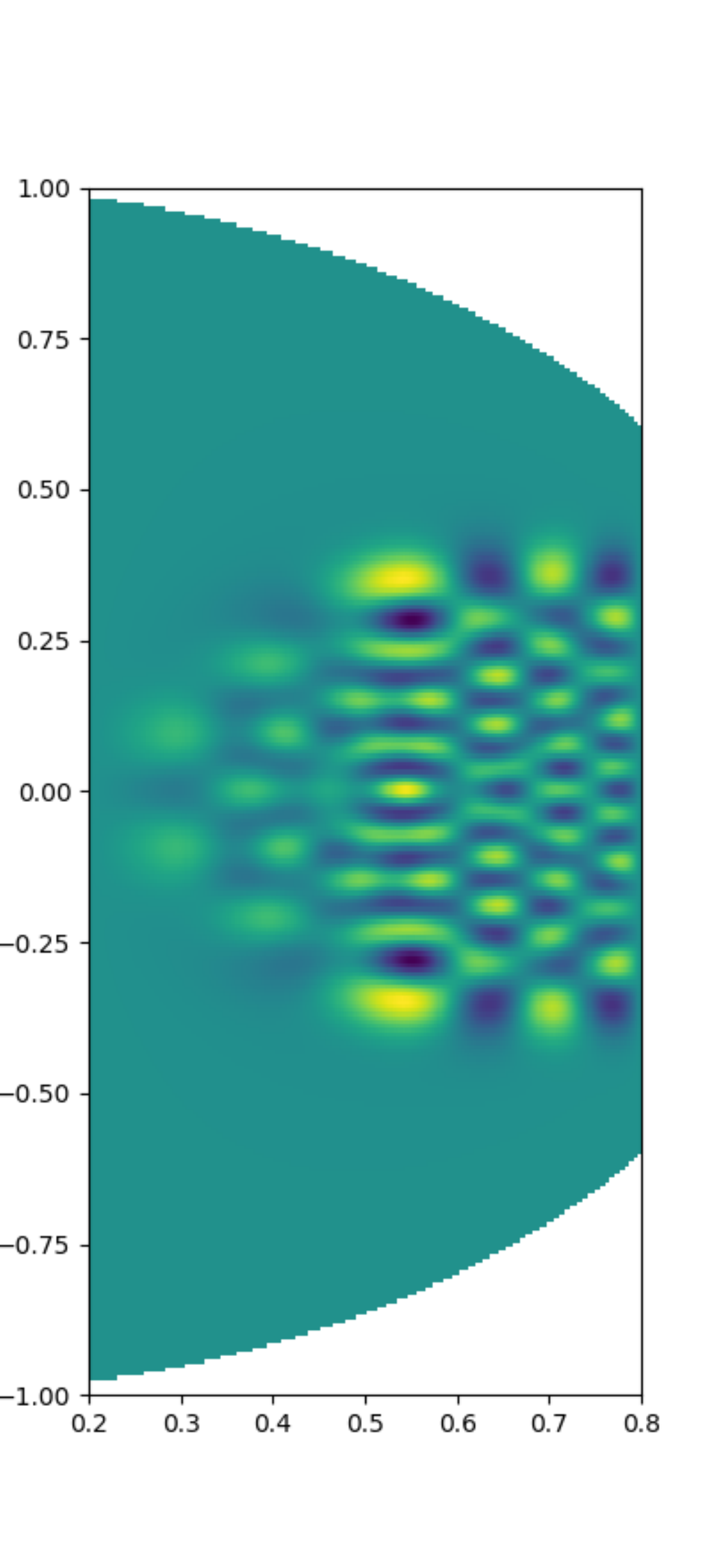}
	\end{subfigure}
	\begin{subfigure}{0.5\textwidth}
	\includegraphics[scale=0.5]{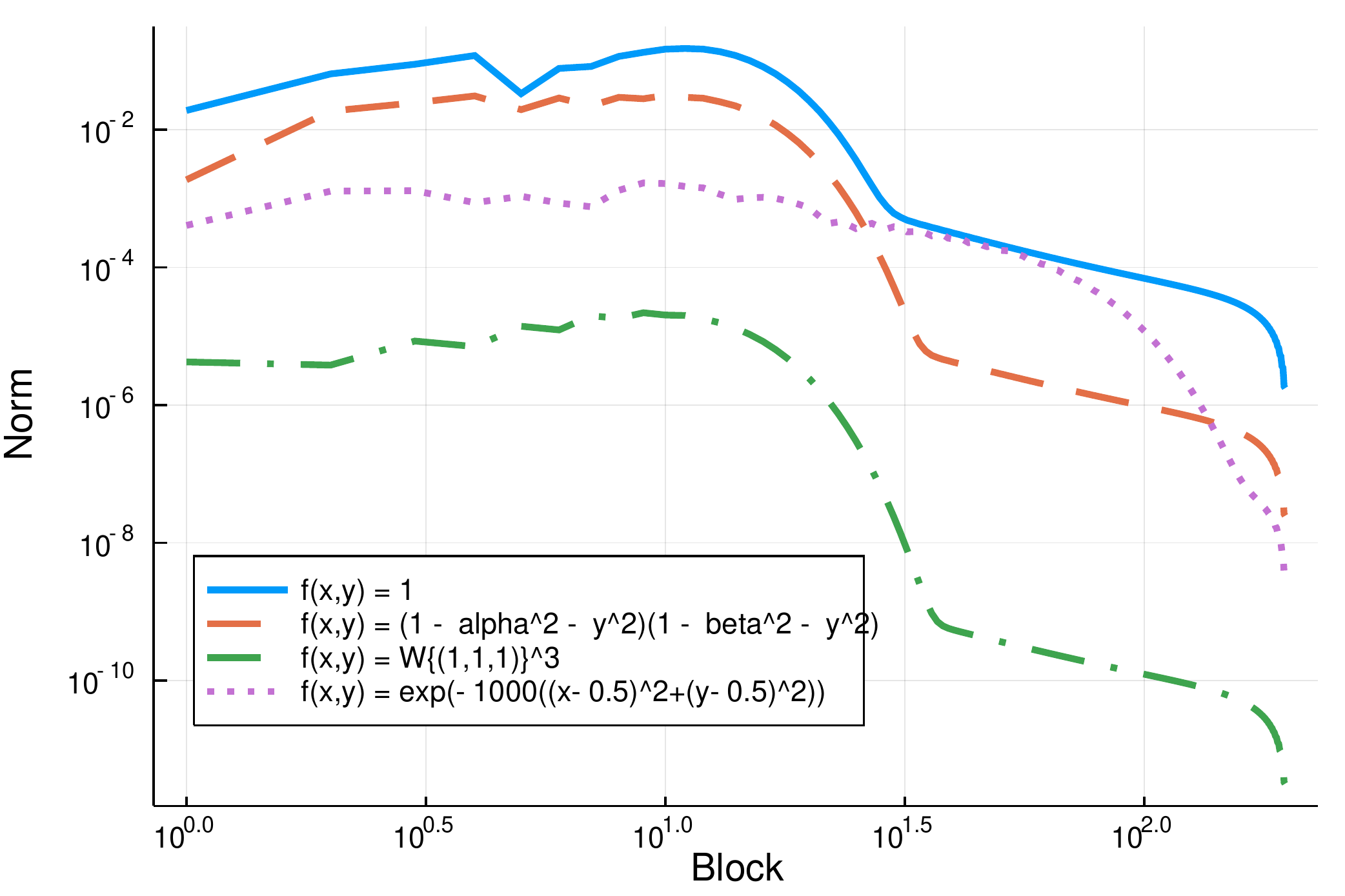}
	\centering
	\end{subfigure}
	\caption{Left: The computed solution to $\Delta u + k^2 \: v \: u = f$ with zero boundary conditions with $f(x,y) = x(1-x^2-y^2)e^x$, $v(x,y) = 1 - (3(x-1)^2 + 5y^2)$ and $k = 100$. Right: The norms of each block of the computed solution of the Helmholtz equation with the given right hand side functions, with $k=20$ and $v(x,y) = 1 - (3(x-1)^2 + 5y^2)$.}
	\centering
	\label{fig:helmholtz}
\end{figure}

\begin{figure}[t]
	\begin{subfigure}{0.3\textwidth}
	\centering
	\includegraphics[scale=0.3]{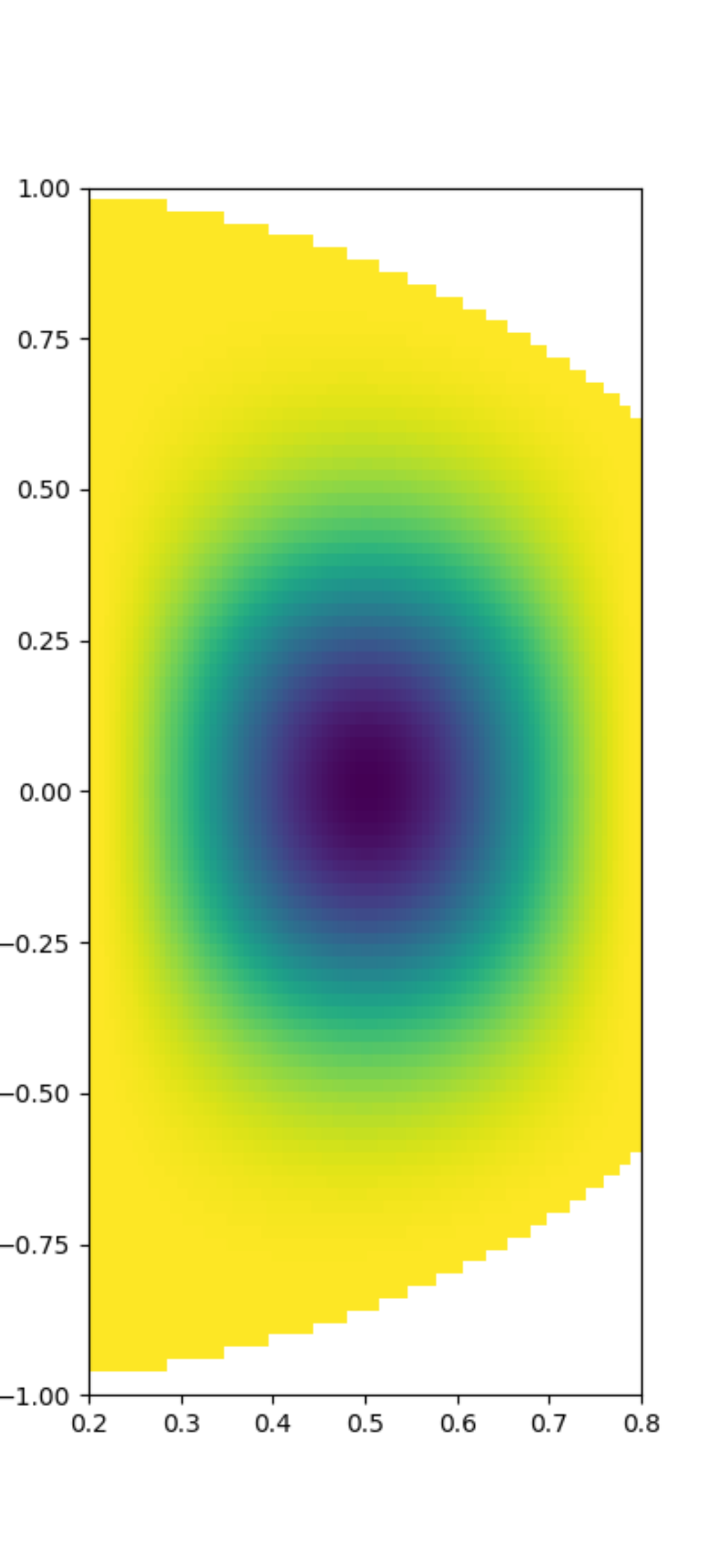}
	\end{subfigure}
	\begin{subfigure}{0.5\textwidth}
	\includegraphics[scale=0.5]{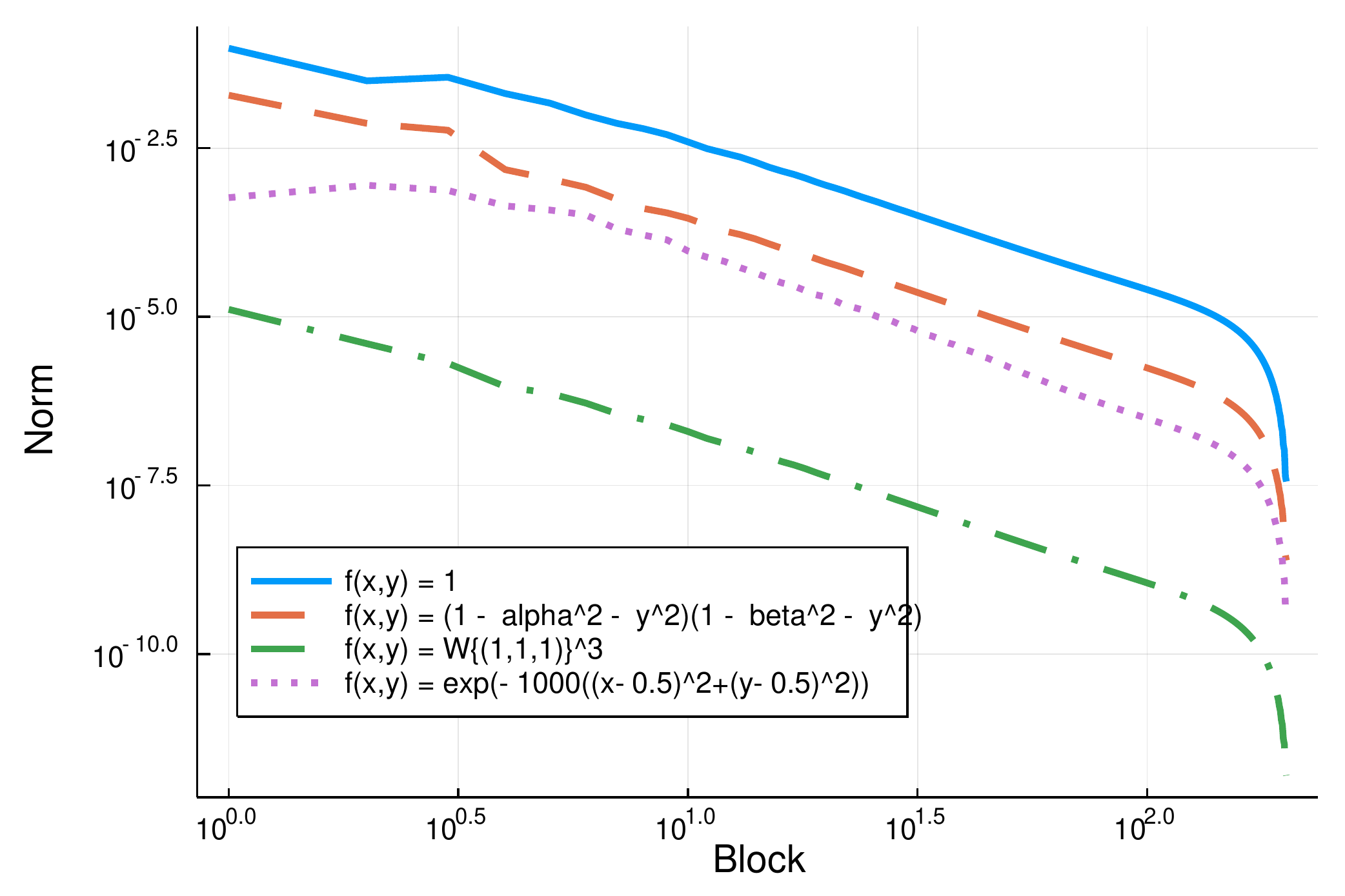}
	\centering
	\end{subfigure}
	\caption{Left: The computed solution to $\Delta^2 u = f$ with zero Dirichlet and Neumann boundary conditions with $f(x,y) = 1 + \text{erf}(5(1 - 10((x - 0.5)^2 + y^2)))$. Right: The norms of each block of the computed solution of the biharmonic equation with the given right hand side functions.}
	\centering
	\label{fig:biharmonic}
\end{figure}

The Poisson equation is the classic problem of finding \(u(x,y)\) given a function \(f(x,y)\) such that:
\begin{align}
	\begin{cases}
    		\Delta u(x,y) = f(x,y) \quad \text{in } \Omega \\
		u(x,y) = 0 \quad \text{on } \partial \Omega
	\end{cases}.
	\label{eqn:poisson}
\end{align}
noting the imposition of zero Dirichlet boundary conditions on $u$.

We can tackle the problem as follows. Denote the coefficient vector for expansion of $u$ in the $\bigWiii$ OP basis up to degree $N$ by $\mathbf{u}$, and the coefficient vector for expansion of $f$ in the $\bighdopiii$ OP basis up to degree $N$ by $\mathbf{f}$. Since $f$ is known, we can obtain $\mathbf{f}$ using the quadrature rule above. In matrix-vector notation, our system hence becomes:
\begin{align*}
    \laplacewiii \mathbf{u} = \mathbf{f}
\end{align*}
which can be solved to find $\mathbf{u}$.
In Figure \ref{fig:poisson} we see the solution to the Poisson equation with zero boundary conditions given in (\ref{eqn:poisson}) in the disk-slice $\Omega$. In Figure \ref{fig:poisson} we also show the norms of each block of calculated coefficients of the approximation for four right-hand sides of the Poisson equation with N = 200, that is, 20,301 unknowns. The rate of decay in the coefficients is a proxy for the rate of convergence of the computed solution: as typical of spectral methods, we expect the numerical scheme to converge at the same rate as the coefficients decay. We see that we achieve algebraic convergence for the first three examples, noting that for right hand-sides that vanish at the corners of our disk-slice ($x\in\{\alpha,\beta\}, \: y = \pm \rho(x)$) we observe faster convergence. 


In Figure \ref{fig:poissonexact} we see an example where the solution calculated to the Poisson equation is shown together with a plot of the exact solution and the error. The example was chosen so that the exact solution was $u(x,y) = \Wiii (x,y) y^3 \exp(x)$, and thus the RHS function $f$ would be $f(x,y) = \Delta \big[\Wiii (x,y) y^3 \exp(x)\big]$. We see that the computed solution is almost exact.

\subsection{Inhomogeneous variable-coefficient Helmholtz}

Find \(u(x,y)\) given functions $v$, $f : \Omega \to \R$ such that:
\begin{align}
	\begin{cases}
    		\Delta u(x,y) + k^2 \: v(x,y) \; u(x,y) = f(x,y) \quad \text{in } \Omega \\
		u(x,y) = 0 \quad \text{on } \partial \Omega
	\end{cases}.
	\label{eqn:helmholtz}
\end{align}
where $k \in \R$, noting the imposition of zero Dirichlet boundary conditions on $u$.

We can tackle the problem as follows. Denote the coefficient vector for expansion of $u$ in the $\bigWiii$ OP basis up to degree $N$ by $\mathbf{u}$, and the coefficient vector for expansion of $f$ in the $\bighdopiii$ OP basis up to degree $N$ by $\mathbf{f}$. Since $f$ is known, we can obtain  the coefficients $\mathbf{f}$ using the quadrature rule above. We can obtain the matrix operator for the variable-coefficient function $v(x,y)$ by using the Clenshaw algorithm with matrix inputs as the Jacobi matrices ${J_x^{(0,0,0)}}^\top, {J_y^{(0,0,0)}}^\top$, yielding an operator matrix of the same dimension as the input Jacobi matrices a la the procedure introduced in \cite{olver2019triangle}. We can denote the resulting operator acting on coefficients in the $\bighdopooo$ space by $V({J_x^{(0,0,0)}}^\top, {J_y^{(0,0,0)}}^\top)$. In matrix-vector notation, our system hence becomes:
\begin{align*}
    (\laplacewiii + k^2 T^{(0,0,0)\to(1,1,1)} \: V({J_x^{(0,0,0)}}^\top, {J_y^{(0,0,0)}}^\top) \: T_W^{(1,1,1)\to(0,0,0)}) \mathbf{u} = \mathbf{f}
\end{align*}
which can be solved to find $\mathbf{u}$. We can see the sparsity and structure of this matrix system in Figure \ref{fig:sparsity} with $v(x,y) = xy^2$ as an example. In Figure \ref{fig:helmholtz} we see the solution to the inhomogeneous variable-coefficient Helmholtz equation with zero boundary conditions given in (\ref{eqn:helmholtz}) in the half-disk $\Omega$, with $k=100$, $v(x,y) = 1 - (3(x-1)^2 + 5y^2)$ and $f(x,y) = x(1-x^2-y^2)e^x$. In Figure \ref{fig:helmholtz} we also show the norms of each block of calculated coefficients of the approximation for four right-hand sides of the inhomogeneous variable-coefficient Helmholtz equation with $k=20$ and $v(x,y) = 1 - (3(x-1)^2 + 5y^2)$ using N = 200, that is, 20,301 unknowns. The rate of decay in the coefficients is a proxy for the rate of convergence of the computed solution. We see that we achieve algebraic convergence for the first three examples, noting that for right hand sides that vanish at the corners of our disk-slice ($x\in\{\alpha,\beta\}, \: y = \pm \rho(x)$) we see faster convergence.


We can extend this to constant non-zero boundary conditions by simply noting that the problem 
\begin{align*}
	\begin{cases}
    		\Delta u(x,y) + k^2 \: v(x,y) \; u(x,y) = f(x,y) \quad \text{in } \Omega \\
		u(x,y) = c \in \R \quad \text{on } \partial \Omega
	\end{cases}
\end{align*}
is equivalent to letting $u = \tilde{u} + c$ and solving
\begin{align*}
	\begin{cases}
    		\Delta \tilde{u}(x,y) + k^2 \: v(x,y) \; \tilde{u}(x,y) = f(x,y) - c \: k^2 \: v(x,y) \; =: g(x,y)  \quad \text{in } \Omega \\
		\tilde{u}(x,y) = 0 \quad \text{on } \partial \Omega.
	\end{cases}.
\end{align*}

\subsection{Biharmonic equation}

Find \(u(x,y)\) given a function \(f(x,y)\) such that:
\begin{align}
	\begin{cases}
    		\Delta^2 u(x,y) = f(x,y) \quad \text{in } \Omega \\
		u(x,y) = 0, \quad \frac{\partial u}{\partial n}(x,y) = 0 \quad \text{on } \partial \Omega
	\end{cases}.
	\label{eqn:biharmonic}
\end{align}
where $\Delta^2$ is the Biharmonic operator, noting the imposition of zero Dirichlet and Neumann boundary conditions on $u$. In Figure \ref{fig:biharmonic} we see the solution to the Biharmonic equation (\ref{eqn:biharmonic}) in the disk-slice $\Omega$. In Figure \ref{fig:biharmonic} we also show the norms of each block of calculated coefficients of the approximation for four right-hand sides of the biharmonic equation with N = 200, that is, 20,301 unknowns.  We see that we achieve algebraic convergence for the first three examples, noting that for right hand sides that vanish at the corners of our disk-slice ($x\in\{\alpha,\beta\}, \: y = \pm \rho(x)$) we see faster convergence.


%
\section{Conclusions}

We have shown that bivariate orthogonal polynomials can lead to sparse discretizations of general linear PDEs on specific domains whose boundary is specified by an algebraic curve---notably here the disk-slice---with Dirichlet boundary conditions. This work extends  the triangle case \cite{beuchler2006new,li2010optimal,olver2019triangle} to non-classical geometries, and forms a building block in developing an $hp-$finite element method to solve PDEs on other polygonal domains by using suitably shaped elements, for example, by dividing the disk into disk slice elements. 
This work serves as a stepping stone to constructing similar methods to solve partial differential equations on 3D sub-domains of the sphere, such as spherical caps and spherical triangles. In particular, orthogonal polynomials (OPs) in cartesian coordinates ($x$, $y$, and $z$) on a half-sphere can be represented using two families of OPs on the half-disk, see \cite[Theorem 3.1]{olver2018orthogonal} for a similar construction of OPs on an arc in 2D, and it is clear from the construction in this paper that discretizations of spherical gradients and Laplacian's are sparse on half-spheres and other suitable sub-components of the sphere. The resulting sparsity in high-polynomial degree discretizations presents an attractive alternative to methods based on bijective mappings (e.g., \cite{DGShallowWater,FEMShallowWater,boyd2005sphere}). Constructing these sparse spectral methods for surface PDEs on half-spheres, spherical caps, and spherical triangles is future work, and has applications in weather prediction \cite{staniforth2012horizontal}. Other extensions include a full $hp$-finite element method on sections of a disk, which has applications in turbulent pipe flow.

\noindent{\bf Acknowledgements}:  We would like to thank the anonymous referees for their helpful comments.  The second author was supported in part by a Leverhulme Trust Research Grant.

\appendix
\section{P-finite element methods using sparse operators}\label{Appendix:PFEM}

\begin{figure}[t]
	\begin{subfigure}{0.3\textwidth}
	\includegraphics[scale=0.28]{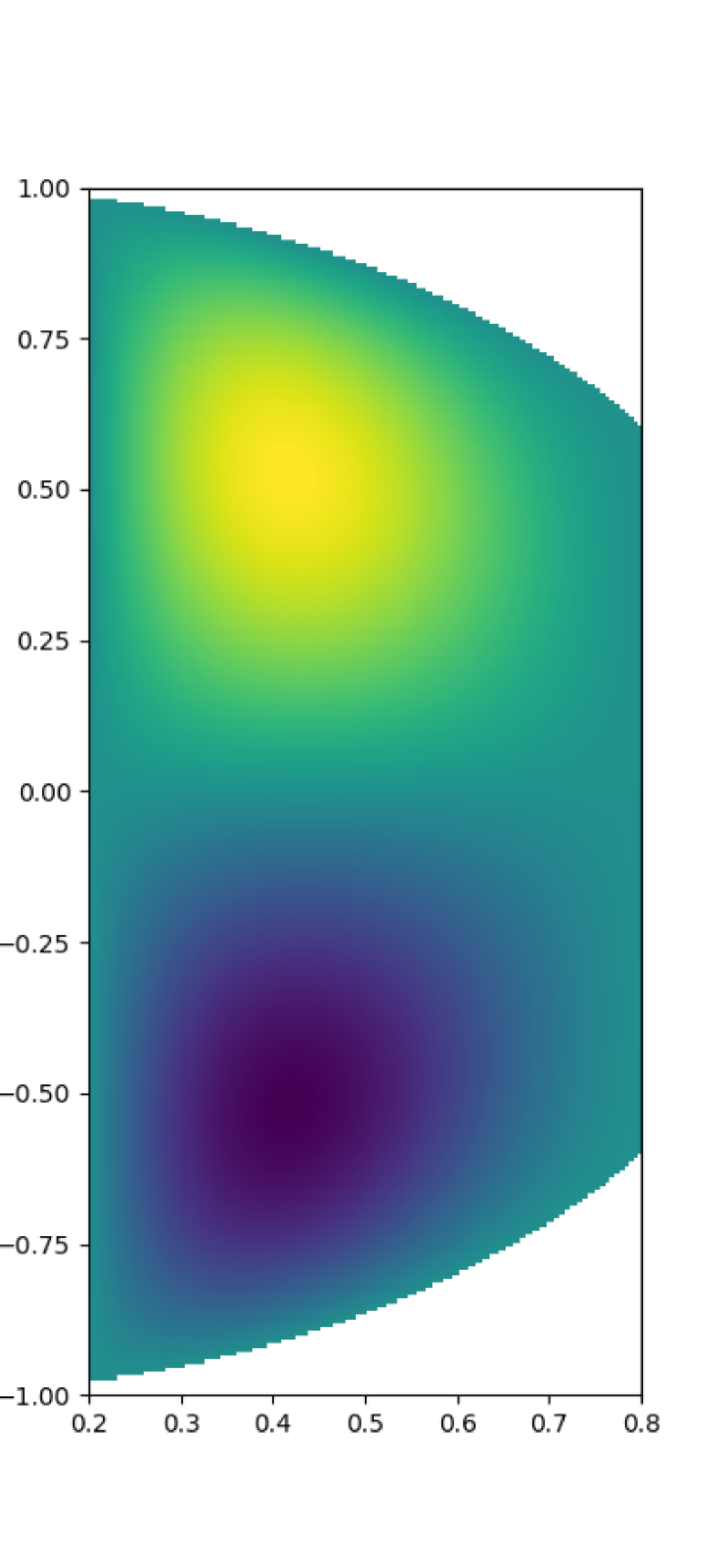}
	\centering
	\end{subfigure}
	\begin{subfigure}{0.3\textwidth}
	\includegraphics[scale=0.3]{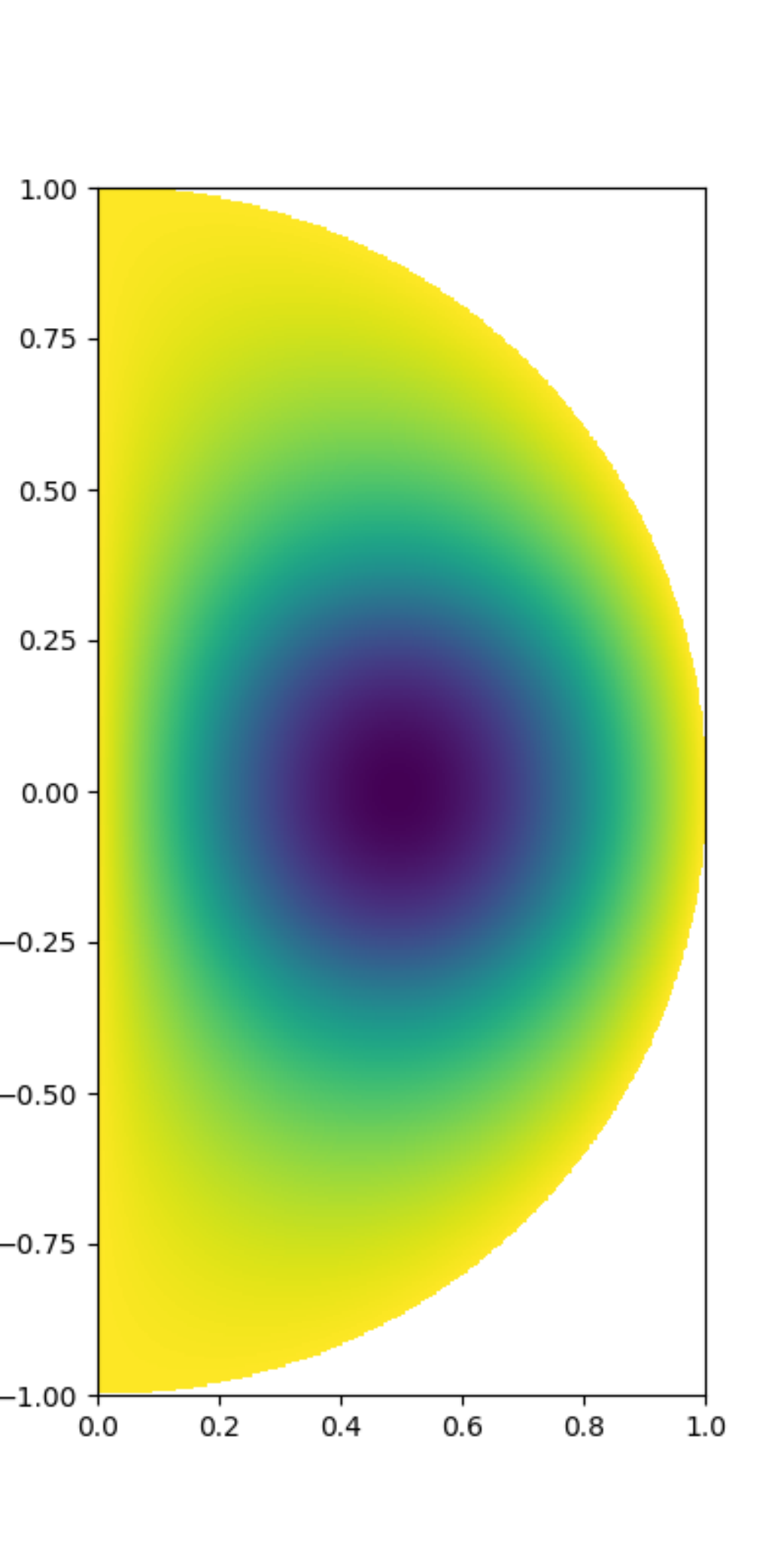}
	\centering
	\end{subfigure}
	\begin{subfigure}{0.3\textwidth}
	\includegraphics[scale=0.3]{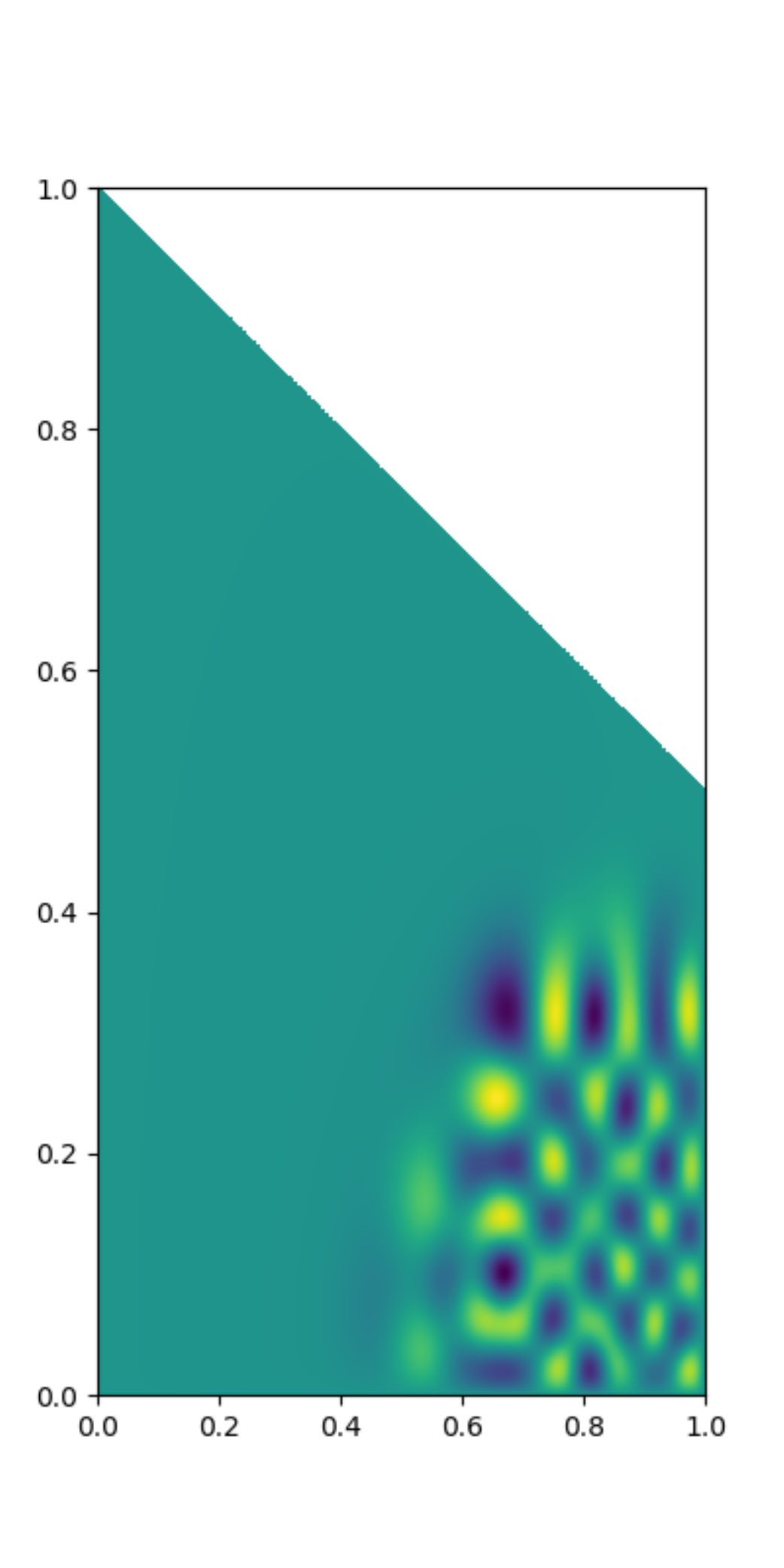}
	\centering
	\end{subfigure}
	\centering
	\caption{Left: The computed solution to $\Delta u = f$ with zero boundary conditions with $f(x,y) = \Wiii(x,y) y \cos(x)$ in the disk-slice using the $p$-FEM approach with a single element. Centre: The computed solution to $\Delta u = f$ with zero boundary conditions with $f(x,y) = 1 + \text{erf}(5(1 - 10((x - 0.5)^2 + y^2)))$ in the half-disk. Right: The computed solution to $\Delta u + k^2 \: v \: u = f$ with zero boundary conditions with $f(x,y) = (1-x) \: x \: y \: (1- \half x - y) \: e^x$, $v(x,y) = 1 - (3(x-1)^2 + 5y^2)$ and $k = 100$. in the trapezium.}
	\centering
	\label{fig:appendixfigs}
\end{figure}

We follow the method of \cite{beuchler2006new} to construct a sparse $p$-finite element method in terms of the operators constructed above, with the benefit of ensuring that the resulting discretisation is symmetric. Consider the 2D Dirichlet problem on a domain $\Omega$:
\begin{align*}
	\begin{cases}
         - \Delta u(x,y) = f(x,y) \quad \text{in } \Omega \\
         u = 0 \quad \text{on } \partial \Omega
         \end{cases}
\end{align*}
This has the weak formulation for any test function $v \in V := H_0^1(\Omega) = \{v \in H^1(\Omega) \quad | \quad v|_{\partial \Omega} = 0 \}$,
\begin{align*}
	L(v) := \int_\Omega f \: v \: d\mathbf{x} = \int_\Omega \nabla u \cdot \nabla v \: d\mathbf{x} =: a(u,v).
\end{align*}

In general, we would let $\FEset$ be the set of elements $\element$ that make up our finite element discretisation of the domain, where each $\element$ is a trapezium or disk slice for example. 

In this section, we limit our discretisation to a single element, that is we let $\element = \Omega$ for a disk-slice domain. We can choose our finite dimensional space $V_p = \{v_p \in V \quad | \quad {\rm deg}\,(v_p|_\element) \le p\}$ for some $p \in \N$.

We seek $u_p \in V_p$ s.t.
\begin{align}
	L(v_p) = a(u_p,v_p) \quad \forall \: v_p \in V_p.
	\label{eqn:FEMweakform}
\end{align}

Define $\Lambda^{(a,b,c)} :=  \ip< \bighdop^{(a,b,c)}, \: {\bighdop^{(a,b,c)}}^\top >_{\Wabc}$ where $\Wabc$ is the weight with which the OPs in $\bighdop^{(a,b,c)}$ are orthogonal with respect to. Note that due to orthogonality this is a diagonal matrix. We can choose a basis for $V_p$ by using the weighted orthogonal polynomials on $\element$ with parameters $a = b = 1$:
\begin{align*}
\bigWiii(x,y) &:= \begin{pmatrix}
		\bigWiii_0(x,y) \\
		\bigWiii_1(x,y) \\
		\bigWiii_2(x,y) \\
		\vdots \\
		\bigWiii_p(x,y)
	\end{pmatrix}, \\
\bigWiii_n(x,y) &:=  \Wiii(x,y) \begin{pmatrix}
						\hdopiii_{n,0}(x,y) \\
						\vdots \\
						\hdopiii_{n,n}(x,y)
					\end{pmatrix} \in \R^{n+1} \quad \forall n = 0,1,2,\dots,p,
\end{align*}
and rewrite (\ref{eqn:FEMweakform}) in matrix form:
\begin{align*}
	a(u_p,v_p) &= \int_\element \nabla u_p \cdot \nabla v_p \: \D \mathbf{x} \\
	&= \int_\element \begin{pmatrix}
					\partial_x v_p \\
					\partial_y v_p
				\end{pmatrix}^\top 
				\begin{pmatrix}
					\partial_x u_p \\
					\partial_y u_p
				\end{pmatrix} \: \D \mathbf{x}
				\\
	&= \int_\element \begin{pmatrix}
					\bighdopooo^\top \Wiii_x \mathbf{v} \\
					\bighdopooo^\top T_W^{(1,1,0)\to(0,0,0)} \Wiii_y \mathbf{v}
				\end{pmatrix}^\top 
				\begin{pmatrix}
					\bighdopooo^\top \Wiii_x \mathbf{u} \\
					\bighdopooo^\top T_W^{(1,1,0)\to(0,0,0)} \Wiii_y \mathbf{u}
				\end{pmatrix} \: \D \mathbf{x}
				\\
	&= \int_\element \Big( \mathbf{v}^\top {\Wiii_x}^\top \bighdopooo \bighdopooo^\top \Wiii_x \mathbf{u} \nonumber \\
					& \quad \quad \quad + \mathbf{v}^\top ({T_W^{(1,1,0)\to(0,0,0)} \Wiii_y})^\top \bighdopooo \bighdopooo^\top T_W^{(1,1,0)\to(0,0,0)} \Wiii_y \mathbf{u}  \Big) \: \D \mathbf{x} \\
	&= \mathbf{v}^\top \: \Big( {\Wiii_x}^\top \Lambda^{(0,0,0)} \Wiii_x \\
	& \quad \quad \quad \quad + ({T_W^{(1,1,0)\to(0,0,0)} \Wiii_y})^\top \Lambda^{(0,0,0)} T_W^{(1,1,0)\to(0,0,0)} \Wiii_y \Big) \: \mathbf{u}
\end{align*}
where $\mathbf{u}, \mathbf{v}$ are the coefficient vectors of the expansions of $u_p, v_p \in V_p$ respectively in the $V_p$ basis ($\bigWiii$ OPs), and
\begin{align*}
	L(v_p) &= \int_\element \: v_p \: f \: \D \mathbf{x}\\
	&= \int_\element \: \mathbf{v}^\top \: \bigWiii \: \bighdopiii^\top \: \mathbf{f} \: \D \mathbf{x} \\
	&= \mathbf{v}^\top \: \ip< \bighdopiii, {\bighdopiii}^\top >_{\Wiii} \: \D \mathbf{x} \\
	&= \mathbf{v}^\top \Lambda^{(1,1,1)} \: \mathbf{f},
\end{align*}
where $\mathbf{f}$ is the coefficient vector for the expansion of the function $f(x,y)$ in the $\bighdopiii$ OP basis.

Since (\ref{eqn:FEMweakform}) is equivalent to stating that
\begin{align*}
	L(\Wiii \hdopiii_{n,k}) = a(u_p,\Wiii \hdopiii_{n,k}) \quad \forall \: n = 0,\dots,p, \: k = 0,\dots,n,
\end{align*}
(i.e. holds for all basis functions of $V_p$) by choosing $v_p$ as each basis function, we can equivalently write the linear system for our finite element problem as:
\begin{align*}
	A\mathbf{u} = \tilde{\mathbf{f}}.
\end{align*}
where the (element) stiffness matrix $A$ is defined by 
\begin{align*}
	A = {\Wiii_x}^\top \Lambda^{(0,0,0)} \Wiii_x + ({T_W^{(1,1,0)\to(0,0,0)} \Wiii_y})^\top \Lambda^{(0,0,0)} T_W^{(1,1,0)\to(0,0,0)} \Wiii_y
\end{align*}
and the load vector $\tilde{\mathbf{f}}$ is given by 
\begin{align*}
	\tilde{\mathbf{f}} = \Lambda^{(1,1,1)} \: \mathbf{f}.
\end{align*}

Note that since we have sparse operator matrices for partial derivatives and basis-transform, we obtain a symmetric sparse (element) stiffness matrix, as well as a sparse operator matrix for calculating the load vector (rhs).

\section{End-Disk-Slice}\label{Appendix:HalfDisk}

The work in this paper on the disk-slice can be easily transferred to the special-case domain of the end-disk-slice , such as half disks, by which we mean
\begin{align*}
	\Omega := \{(x,y) \in \R^2 \quad | \quad \alpha < x < \beta, \: \gamma \rho(x) < y < \delta \rho(x)\}
\end{align*}
with
\begin{align*}
\begin{cases}
\alpha &\in (0,1) \\
\beta &:= 1 \\
(\gamma, \delta) &:= (-1,1) \\
\rho(x) &:= (1-x^2)^{\half}.
\end{cases}
\end{align*}
Our 1D weight functions on the intervals $(\alpha, \beta)$ and $(\gamma, \delta)$ respectively are then given by:
\begin{align*}
\begin{cases}
\genjacw^{(a,b)}(x) &:= (x - \alpha)^{a} \: \rho(x)^{b} \\
\jacw^{(a)}(x) &:= (1-x^2)^b.
\end{cases}
\end{align*}

Note here how we can remove the need for third parameter, which is why we consider this a special case. This will make some calculations easier, and the operator matrices more sparse. The weight $\jacw^{(b)}(x)$ is a still the same ultraspherical weight (and the corresponding OPs are the Jacobi polynomials $\{\jac_n^{(b, b)}\}$). $\genjacw^{(a,b)}(x)$ is the (non-classical) weight for the OPs denoted $\{\genjac_n^{(a,b)}\}$. Thus we arrive at the two-parameter family of 2D orthogonal polynomials $\{\hdopnkab\}$ on $\Omega$ given by, for \(0 \le k \le n, \: n = 0,1,2,\dots,\)
\begin{align*}
	\hdopnkab(x,y) := \genjacnmk^{(a, 2b+2k+1)}(x) \: \rho(x)^k \: \jac_k^{(b,b)}\fpr(\frac{y}{\rho(x)}), \quad (x,y) \in \Omega, 
\end{align*}
orthogonal with respect to the weight
\begin{align*}
	\Wab(x,y) &:= \genjacw^{(a,2b)}(x) w^{(b)}_P\fpr(\frac{y}{\rho(x)}) \nonumber \\
	&= (x - \alpha)^{a} \: (\rho(x)^2 -y^2)^b \nonumber \\
	&= (x - \alpha)^{a} \: (1 - x^2 -y^2)^b , \quad (x,y) \in \Omega.
\end{align*}

The sparsity of operator matrices for partial differentiation by $x, y$ as well as for parameter transformations generalise to such end-disk-slice domains. For instance, if we inspect the proof of Lemma \ref{theorem:sparsityofdifferentialoperators}, we see that it can easily generalise to the weights and domain $\Omega$ for an end-disk-slice.

In Figure \ref{fig:appendixfigs} we see the solution to the Poisson equation with zero boundary conditions in the half-disk $\Omega$ with $(\alpha,\beta) := (0,1)$.

\section{Trapeziums}\label{Appendix:Trapezium}

We can further extend this work to trapezium shaped domains. Note that for any trapezium there exists an affine map to the canonical trapezium domain that we consider here, given by
\begin{align*}
	\Omega := \{(x,y) \in \R^2 \quad | \quad \alpha < x < \beta, \: \gamma \rho(x) < y < \delta \rho(x)\}
\end{align*}
with 
\begin{align*}
\begin{cases}
(\alpha, \beta) &:= (0,1) \\
(\gamma, \delta) &:= (0,1) \\
\rho(x) &:= 1- \xi x, \quad \xi \in (0,1) \\
\genjacw^{(a,b,c)}(x) &:= (\beta - x)^a \: (x - \alpha)^{b} \: \rho(x)^{c} = (1-x)^{a} \: x^b \: (1-\xi x)^{c} \\
\jacw^{(a,b)}(x) &:= (\delta-x)^a \: (x - \gamma)^b = (1-x)^a \: x^b.
\end{cases}
\end{align*}
The weight $\jacw^{(a,b)}(x)$ is the weight for the shifted Jacobi polynomials on the interval $[0,1]$, and hence the corresponding OPs are the shifted Jacobi polynomials $\{\tilde{P}_n^{(a, b)}\}$. We note that the shifted Jacobi polynomials relate to the normal Jacobi polynomials by the relationship $\tilde{P}_n^{(a,b)}(x) = \jac_n^{(a,b)}(2x-1)$ for any degree $n = 0,1,2,\dots$ and $x \in [0,1]$. $\genjacw^{(a,b,c)}(x)$ is the (non-classical) weight for the OPs we dentote $\{\genjac_n^{(a,b,c)}\}\). Thus we arrive at the four-parameter family of 2D orthogonal polynomials $\{\hdopnk^{(a,b,c,d)}\}$ on $\Omega$ given by, for \(0 \le k \le n, \: n = 0,1,2,\dots,\)
\begin{align*}
	\hdopnk^{(a,b,c,d)}(x,y) := \genjacnmk^{(a, b, c+d+2k+1)}(x) \: \rho(x)^k \: \tilde{P}_k^{(d,c)}\fpr(\frac{y}{\rho(x)}), \quad (x,y) \in \Omega, 
\end{align*}
orthogonal with respect to the weight
\begin{align*}
	W^{(a,b,c,d)}(x,y) &:= \genjacw^{(a,b,c+d)}(x) \: w^{(d,c)}_P\fpr(\frac{y}{\rho(x)}) \nonumber \\
	&= (1-x)^a \: x^b \: y^c \: (1- \xi x - y)^d, \quad (x,y) \in \Omega.
\end{align*}

In Figure \ref{fig:appendixfigs} we see the solution to the Helmholtz equation with zero boundary conditions in the trapezium $\Omega$ with $\xi := \half$.

\bibliography{halfdisk}

\end{document}